\documentclass[11pt]{amsart}
\usepackage[a4paper,margin=2.5cm]{geometry}

\usepackage[english]{babel}
\usepackage{graphicx,amsmath,amssymb,color, enumerate}
\usepackage{tikz}
\usetikzlibrary{arrows.meta,calc,angles,quotes}
\numberwithin{equation}{section}
\usepackage{enumerate}
\usepackage{comment}
\usepackage{xcolor}
\definecolor{brickred}{rgb}{0.8, 0.25, 0.33}
\definecolor{blue(ryb)}{rgb}{0.01, 0.28, 1.0}
\definecolor{brandeisblue}{rgb}{0.0, 0.44, 1.0}
\definecolor{ceruleanblue}{rgb}{0.16, 0.32, 0.75}
\definecolor{cobalt}{rgb}{0.0, 0.28, 0.67}
\definecolor{coolblack}{rgb}{0.0, 0.18, 0.39}
\definecolor{darkblue}{rgb}{0.0, 0.0, 0.55}
\usepackage[hidelinks]{hyperref}

\usepackage{placeins}
\usepackage{needspace}

\hypersetup{
	colorlinks,
	citecolor=darkblue,
	filecolor=black,
	linkcolor=darkblue,
	urlcolor=black
}

\usepackage{comment}

\newtheorem{theorem}{Theorem}[section]
\newtheorem{lemma}[theorem]{Lemma}
\newtheorem{proposition}[theorem]{Proposition}
\newtheorem{corollary}[theorem]{Corollary}

\theoremstyle{remark}

\newtheorem{remark}[theorem]{Remark}

\newtheorem{definition}[theorem]{Definition\rm}
\newtheorem{example}{\it Example\/}

\newcommand{\mn}{\medskip\noindent}

\newcommand{\R}{\mathbb{R}}

\newcommand{\Id}{\mathsf{Id}}

\newcommand{\diff}{\,\mathrm{d}}
\newcommand{\ii}{\,\mathrm{i}}

\title[Inverse semiclassical scattering at fixed energy]
{Inverse Semiclassical Scattering at Fixed Energy}

\author[F. Nicoleau]{Fran\c{c}ois Nicoleau}
\address[F. Nicoleau]{Laboratoire de Math\'ematiques Jean Leray, CNRS, Nantes Universit\'e,
44000 Nantes, France.}
\email{francois.nicoleau@univ-nantes.fr}


\begin{document}
	
	\maketitle

\begin{center}
	\emph{Dedicated to Didier Robert, for his major contributions to
		semiclassical analysis, \\ scattering theory, and coherent states.}
\end{center}

\begin{abstract}
	We investigate inverse scattering at a fixed energy for semiclassical
	Schrödinger operators with smooth potentials. For compactly supported potentials satisfying a natural virial
	condition, we prove that, under the sharp condition, 
	$\liminf_{h\to0}\|S_1(\lambda,h)-S_2(\lambda,h)\|<2$
	implies $V_1=V_2$, provided the fixed energy $\lambda$ lies above
	both potentials. 
	We then consider radial short-range repulsive potentials.
	Under a monotonicity assumption on the radial force, we show that the
	semiclassical differential cross section at a fixed energy, up
	to $o(1)$ as $h\to0$, determines the potential  in the
	classically accessible region. Finally, we obtain an analogous rigidity
	result for nontrapping compactly supported perturbations of the Euclidean
	metric under a strict convexity assumption.
\end{abstract}

\section{Introduction}\label{sec:intro}

The purpose of this paper is to study the inverse problem for
semiclassical scattering at a fixed energy. More precisely, we investigate
to what extent the scattering data in the semiclassical limit determine
the underlying potential. Let
\begin{equation}\label{eq:Ph}
	H(h)=-\frac{h^2}{2}\Delta+V(x),\qquad 0<h\leq1,
\end{equation}
be the semiclassical Schr\"odinger operator in $L^2(\mathbb R^n)$,
$n\geq3$. Throughout the paper, $V$ is real-valued and smooth. We shall
work with short-range potentials satisfying, for some $\mu>1$,
\begin{equation}\label{eq:short-range}
	|\partial_x^\alpha V(x)|
	\leq C_\alpha\langle x\rangle^{-\mu-|\alpha|},
	\qquad \alpha\in\mathbb N^n,
	\qquad
	\langle x\rangle=(1+|x|^2)^{1/2}.
\end{equation}
For such potentials, trajectories escaping to infinity have free
asymptotics, while trapped trajectories may still occur at a given
energy.

\mn
We denote by
\begin{equation}\label{eq:H0}
	H_0(h)=-\frac{h^2}{2}\Delta
\end{equation}
the free Hamiltonian associated with $H(h)$, and by
\begin{equation}\label{eq:resolvent}
	R(z;h)=(H(h)-z)^{-1},
	\qquad z\in\mathbb C\setminus\sigma(H(h)),
\end{equation}
the resolvent of $H(h)$. Let $\mathcal F_h$ denote the semiclassical Fourier transform, defined for
$f\in\mathcal S(\mathbb R^n)$ by
\[
(\mathcal F_h f)(\xi)
=
(2\pi h)^{-n/2}
\int_{\mathbb R^n}e^{-ix\cdot\xi/h}f(x)\,dx.
\]
For \(\lambda>0\), the Fourier restriction operator at energy \(\lambda\) is given by
\[
F_0(\lambda,h)f(\omega)
=
(2\lambda)^{(n-2)/4}
(\mathcal F_h f)(\sqrt{2\lambda}\,\omega),
\qquad \omega\in\mathbb S^{n-1}.
\]
The scattering matrix
\[
S(\lambda,h):L^2(\mathbb S^{n-1})
\longrightarrow L^2(\mathbb S^{n-1})
\]
is unitary and can be written as
\begin{equation}\label{eq:S-T}
	S(\lambda,h)=\Id-2\pi\ii\,T(\lambda,h),
\end{equation}
where
\begin{equation}\label{eq:T-matrix}
	T(\lambda,h)
	=
	F_0(\lambda,h)
	\bigl(V-VR(\lambda+\ii0;h)V\bigr)
	F_0(\lambda,h)^*.
\end{equation}
Here
\[
R(\lambda+\ii0;h)
=
s-\lim_{\varepsilon\downarrow0}\ 
(H(h)-\lambda-\ii\varepsilon)^{-1}
\]
denotes the outgoing boundary value of the resolvent, where the strong
limit is taken in
\[
\mathcal B\bigl(L^2_s(\mathbb R^n),L^2_{-s}(\mathbb R^n)\bigr),
\qquad s>\frac12,
\]
with
\begin{equation}
L^2_s(\mathbb R^n)
=
\bigl\{u\in L^2_{\mathrm{loc}}(\mathbb R^n);
\ \langle x\rangle^s u\in L^2(\mathbb R^n)\bigr\}.
\end{equation}

\mn 
Away from the diagonal, $T(\lambda,h)$ has a smooth kernel, denoted by
\[
T(\theta,\omega;\lambda,h),
\qquad
\theta,\omega\in\mathbb S^{n-1},
\quad \theta\neq\omega.
\]
The corresponding scattering amplitude is defined by Kuroda's formula
\begin{equation}\label{eq:f-T}
	f(\omega\to\theta;\lambda,h)
	=
	c(\lambda,h)\,T(\theta,\omega;\lambda,h),
\end{equation}
where
\begin{equation}\label{eq:c-scattering}
	c(\lambda,h)
	=
	-(2\pi)^{(n+1)/2}
	(2\lambda)^{-(n-1)/4}
	h^{(n-1)/2}
	\exp\left(-\ii\frac{(n-3)\pi}{4}\right).
\end{equation}
For the standard scattering theory recalled above, we refer to
Isozaki--Kitada \cite{IsozakiKitada1,IsozakiKitada2},
Reed--Simon \cite{ReedSimon3}, and Yafaev \cite{Yafaev}.

\mn 
Our main inverse problem is to determine to what extent the sharp condition
\begin{equation}\label{eq:intro-scattering-close}
	\liminf_{h\to0}
	\|S_1(\lambda,h)-S_2(\lambda,h)\|_
	{\mathcal B(L^2(\mathbb S^{n-1}))}
	<2,
\end{equation}
at a fixed energy $\lambda>0$, determines the potential.

\mn 
There is a natural limitation to such a fixed-energy inverse problem.
In \cite{Nicoleau}, it was proved that, for two short-range potentials
$V_1$ and $V_2$, at a fixed non-trapping energy $\lambda>0$ (defined below),
\begin{equation}\label{eq:old-direct-result}
	(\lambda-V_1)_+=(\lambda-V_2)_+
	\quad\Longrightarrow\quad
	S_1(\lambda,h)-S_2(\lambda,h)
	=\mathcal O(h^\infty)
\end{equation}
in the operator norm of $\mathcal B(L^2(\mathbb S^{n-1}))$ as $h\to0$,
where $(a)_+:=\max\{a,0\}$.
Thus, at a fixed energy, one cannot in general expect the semiclassical
scattering data to determine the potential in the classically forbidden
region $\{V\geq\lambda\}$. The natural inverse problem is therefore to
recover $V$ in the classically allowed region $\{V<\lambda\}$, or,
equivalently, to recover the function $(\lambda-V)_+$.

\mn 
The semiclassical direct problem has a precise classical interpretation.
Semiclassical asymptotics for scattering amplitudes were first obtained
by Protas \cite{Protas} and Vainberg \cite{Vainberg} for compactly
supported potentials. Yajima \cite{Yajima} subsequently treated non-compactly
supported short-range potentials, obtaining an asymptotic expansion in
an $L^2$ sense with respect to the energy and angular variables.
Robert and Tamura \cite{RobertTamura} then obtained, for short-range
potentials, a precise fixed-energy description at a non-trapping energy,
expressing the scattering amplitude, for regular outgoing directions, as
a finite sum of contributions associated with the underlying classical
scattering trajectories. Further microlocal descriptions of the semiclassical scattering matrix
were obtained by Alexandrova \cite{Alexandrova}, as we shall discuss
in more detail below.

\mn 
We recall the classical objects entering the semiclassical scattering
asymptotics, using a notation adapted to the present paper.  For the standard semiclassical framework and its relation with the
underlying classical Hamiltonian dynamics, we refer, for instance, to
Robert~\cite{Robert}. The
Hamiltonian associated with~\eqref{eq:Ph} is
\begin{equation}\label{eq:hamiltonian}
	p(x,\xi)=\frac{|\xi|^2}{2}+V(x),
\end{equation}
and the corresponding classical trajectories
$t\mapsto(q_{\rm cl}(t),p_{\rm cl}(t))$ are the integral curves of the
Hamiltonian vector field, that is,
\begin{equation}\label{eq:Hamilton-system}
	\dot q_{\rm cl}(t)=p_{\rm cl}(t),\qquad
	\dot p_{\rm cl}(t)=-\nabla V(q_{\rm cl}(t)).
\end{equation}

\mn
At the energy $\lambda>0$, the energy surface is
\begin{equation}\label{eq:energy-surface}
	\Sigma_\lambda
	=
	\left\{(x,\xi)\in T^*\mathbb R^n;
	\frac{|\xi|^2}{2}+V(x)=\lambda\right\}.
\end{equation}
We recall that $\lambda>0$ is a non-trapping energy if, for every
$R>1$ sufficiently large, there exists $T=T(R)>0$ such that
\begin{equation}\label{eq:nontrapping}
	|q_{\rm cl}(t;y,\eta)|>R,
	\qquad |t|>T,
\end{equation}
whenever $|y|<R$ and $(y,\eta)\in\Sigma_\lambda$, where
$t\mapsto(q_{\rm cl}(t;y,\eta),p_{\rm cl}(t;y,\eta))$ denotes the
solution of~\eqref{eq:Hamilton-system} with initial data $(y,\eta)$ at $t=0$.

\mn
Under the non-trapping assumption and~\eqref{eq:short-range}, every
classical trajectory $(q_{\rm cl}(t),p_{\rm cl}(t))$ on the energy surface $\Sigma_\lambda$ is a scattering trajectory and admits
asymptotic data $(r_\pm,v_\pm)$ such that
\begin{equation}\label{eq:free-asymptotics}
	q_{\rm cl}(t)=v_\pm t+r_\pm+o(1),\qquad
	p_{\rm cl}(t)=v_\pm+o(1),
	\qquad t\to\pm\infty.
\end{equation}
Since the energy is conserved and $V(q_{\rm cl}(t))\to0$ as
$t\to\pm\infty$, one has
$
|v_\pm|=\sqrt{2\lambda}.
$
The map
\begin{equation}\label{eq:classical-scattering-transformation}
	\mathcal S_{\rm cl}:(r_-,v_-)\longmapsto(r_+,v_+)
\end{equation}
is the classical scattering transformation.

\mn 
For later use, we describe the classical scattering transformation in
the coordinates naturally associated with the semiclassical scattering
matrix. Fix an incoming direction $\omega\in\mathbb S^{n-1}$ and set
\begin{equation}\label{eq:impact-plane}
	\Lambda_\omega
	=
	\{z\in\mathbb R^n;\ z\cdot\omega=0\}.
\end{equation}
For $z\in\Lambda_\omega$, let
\[
\bigl(q_\infty(t;z,\lambda),p_\infty(t;z,\lambda)\bigr)
\]
be the scattering trajectory satisfying
\begin{equation}\label{eq:incoming-asymptotics}
	q_\infty(t;z,\lambda)
	=
	\sqrt{2\lambda}\,\omega t+z+o(1),
	\qquad
	p_\infty(t;z,\lambda)
	=
	\sqrt{2\lambda}\,\omega+o(1),
	\qquad t\to-\infty.
\end{equation}
As $t\to+\infty$, there exist smooth functions
$\xi_\infty(z;\lambda)\in\mathbb S^{n-1}$ and
$r_\infty(z;\lambda)\in\mathbb R^n$ such that
\begin{equation}\label{eq:outgoing-asymptotics}
	\begin{split}
		q_\infty(t;z,\lambda)
		&=
		\sqrt{2\lambda}\,
		\xi_\infty(z;\lambda)t+r_\infty(z;\lambda)+o(1),\\
		p_\infty(t;z,\lambda)
		&=
		\sqrt{2\lambda}\,
		\xi_\infty(z;\lambda)+o(1).
	\end{split}
	\qquad t\to+\infty.
\end{equation}
For $\theta\in\mathbb S^{n-1}$, let
\[
\Pi_\theta x=x-(x\cdot\theta)\theta
\]
denote the orthogonal projection onto $\theta^\perp$. The classical
scattering map at energy $\lambda$ is then given by
\begin{equation}\label{eq:classical-scattering-map}
	\kappa_\lambda(\omega,z)
	=
	\left(
	\xi_\infty(z;\lambda),
	\Pi_{\xi_\infty(z;\lambda)}r_\infty(z;\lambda)
	\right).
\end{equation}

\mn
Figure~\ref{fig:classical-scattering} summarizes the geometry. For a
genuinely short-range potential, the interaction is not compactly
supported; the circle only represents the region where it is significant.
The parameter $z$ labels the incoming line, while
$\xi_\infty(z;\lambda)$ is the outgoing direction.

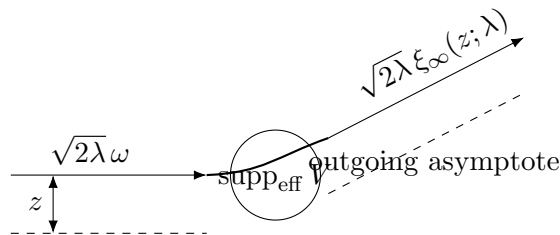
\begin{figure}[ht]
	\centering
	\begin{tikzpicture}[scale=0.65,>=Latex]
		\draw[->] (-5,0) -- (-1.3,0);
		\draw[->] (1.0,0.7) -- (4.7,2.6);
		\draw[dashed] (-5,-1.1) -- (-1.3,-1.1);
		\draw[dashed] (1.0,-0.35) -- (4.7,1.55);
		\draw[thick] (-1.3,0) .. controls (-0.2,0.05) and (0.0,0.35) .. (1.0,0.7);
		\draw (0,0) circle (1.25);
		\node at (0,-0.05) {$\operatorname{supp}_{\rm eff}V$};
		\draw[<->] (-4.2,-1.1) -- (-4.2,0);
		\node[left] at (-4.2,-0.55) {$z$};
		\node[above] at (-3.5,0) {$\sqrt{2\lambda}\,\omega$};
		\node[above,rotate=27] at (3.2,1.85)
		{$\sqrt{2\lambda}\,\xi_\infty(z;\lambda)$};
		\node[below] at (4.0,0.65) {outgoing asymptote};
	\end{tikzpicture}
	\caption{A classical scattering trajectory at the fixed energy $\lambda$.}

	\label{fig:classical-scattering}
\end{figure}

\Needspace{12\baselineskip}

\mn
We first consider compactly supported potentials. At a fixed
non-trapping energy, the condition
$\liminf_{h\to0}\|S_1(\lambda,h)-S_2(\lambda,h)\|_
{\mathcal B(L^2(\mathbb S^{n-1}))}<2$
already implies equality of the classical scattering maps.
Combined with a classical rigidity result, this yields our first
uniqueness theorem.

\begin{theorem}\label{thm:compact}
	Let $V_1,V_2\in C_c^\infty(\mathbb R^n)$, $n\geq3$, and let
	$\lambda>0$ satisfy
	\begin{equation}\label{eq:energy-above}
		\lambda>\sup_{\mathbb R^n}V_j,
		\qquad j=1,2.
	\end{equation}
Assume moreover the virial condition
\begin{equation}\label{eq:virial-intro}
	2(\lambda-V_j(x))-x\cdot\nabla V_j(x)>0,
	\qquad x\in\mathbb R^n,
	\qquad j=1,2.
\end{equation}
	If
	\begin{equation} \label{eq:scattering-equality-compact}
		\liminf_{h\to0}
		\|S_1(\lambda,h)-S_2(\lambda,h)\|_
		{\mathcal B(L^2(\mathbb S^{n-1}))}
		<2,
	\end{equation}
	then
	\begin{equation}\label{eq:compact-uniqueness}
		V_1=V_2.
	\end{equation}
\end{theorem}

\mn
Let us briefly explain the main ingredients of the proof. A fundamental
result of Alexandrova~\cite{Alexandrova} shows, for short-range potentials
and microlocally away from the diagonal, that the semiclassical scattering
matrix is a semiclassical Fourier integral operator associated with the
classical scattering relation. In particular, her work provides a precise
microlocal link between the quantum scattering matrix and the underlying
classical scattering dynamics.

\mn
Here we use instead a particularly convenient description due to
Ingremeau~\cite{Ingremeau}. For compactly supported potentials at a
non-trapping energy, his result describes explicitly the action of the
scattering matrix on Gaussian states localized at points of
$T^*\mathbb S^{n-1}$. Using this propagation result and alternating sums of Gaussian
states, we obtain a quantitative threshold: it is enough that
\begin{equation}
\liminf_{h\to0}
\|S_1(\lambda,h)-S_2(\lambda,h)\|<2
\end{equation}
to recover the classical scattering map.

\mn 
We now recall the precise form of Ingremeau's result that we shall use.
His result is stated at the fixed energy $1/2$. By a harmless rescaling,
we may reduce our problem at energy $\lambda>0$ to this setting. If $\phi_{\rho,h}$ is a Gaussian state localized near
$\rho\in T^*\mathbb S^{n-1}$, then, for every $N\in\mathbb N$,
\begin{equation}\label{eq:Ingremeau-intro}
	S(\lambda,h)\phi_{\rho,h}
	=
	e^{\ii\delta/h}
	\phi_{\kappa_\lambda(\rho),h}^{(N)}
	+\mathcal O(h^N),
\end{equation}
where $\delta\in\mathbb R$ and
$\phi_{\kappa_\lambda(\rho),h}^{(N)}$ is an outgoing Gaussian state,
with an amplitude admitting a finite semiclassical expansion, localized
near the phase point $\kappa_\lambda(\rho)$.

\mn
The essential point for our purposes is the transport of the center of
the Gaussian state from $\rho$ to $\kappa_\lambda(\rho)$. We show that
this implies
\begin{equation}\label{eq:kappa-recovery-intro}
	\liminf_{h\to0}
	\|S_1(\lambda,h)-S_2(\lambda,h)\|_
	{\mathcal B(L^2(\mathbb S^{n-1}))}
	<2
	\quad\Longrightarrow\quad
	\kappa_{1,\lambda}=\kappa_{2,\lambda}.
\end{equation}
The proof uses finite orbit segments of
$\kappa_{1,\lambda}^{-1}\circ\kappa_{2,\lambda}$
and alternating sums of Gaussian states.
Their asymptotic orthogonality gives the threshold $2$;
see Proposition~\ref{lem:recovery-kappa}.

\mn
It remains to pass from the classical scattering map to the potential.
Following the approach used by Alexandrova \cite{Alexandrova}, we
interpret the classical trajectories in terms of the associated Jacobi
metric. Under \eqref{eq:energy-above}, trajectories at energy $\lambda$,
up to reparametrization, are geodesics of the Jacobi metrics
\begin{equation}\label{eq:Jacobi-intro}
	g_j=2(\lambda-V_j)g_0,
	\qquad j=1,2.
\end{equation}
The virial condition~\eqref{eq:virial-intro} implies that $\lambda$ is
a non-trapping energy and that the Euclidean spheres form a strictly
convex foliation for the Jacobi metrics~\eqref{eq:Jacobi-intro}. The former property allows us to apply Ingremeau's result
\cite{Ingremeau}, while the latter allows us to use the rigidity theorem
of Stefanov--Uhlmann--Vasy \cite{SUV2016} and
conclude that $g_1=g_2$, hence $V_1=V_2$.  We also mention that related fixed-energy scattering rigidity results
for simple MP-systems were obtained by Mu\~noz-Thon~\cite{MunozThon2024}.
In the purely potential case considered here, simplicity amounts to
simplicity of the associated Jacobi metric, and his result yields
uniqueness of the potential when the background metric is fixed.
By contrast, our result does not require the Jacobi metric to be simple;
the simplicity assumption is replaced by the convex foliation condition
provided by the virial hypothesis.

\mn
We next turn to radial short-range potentials. Here the inverse problem
has a different character. In particular, the energy is allowed to lie
below the maximum of the potential and, as explained above, one can
only expect to recover the potential in the classically accessible
region.

\mn
Let
\begin{equation}\label{eq:radial-potentials-intro}
	V_j(x)=v_j(|x|),
	\qquad j=1,2,
\end{equation}
be smooth radial short-range potentials satisfying
\begin{equation}\label{eq:radial-assumptions-intro}
	v_j'(r)<0,\qquad r>0,
	\qquad
	v_j(0)>\lambda.
\end{equation}
The assumption $v_j(0)>\lambda$, together with the strict monotonicity
of $v_j$, ensures the existence of a unique positive radial turning
point $r_{j,\lambda}>0$ defined by
\begin{equation}\label{eq:r-lambda-intro}
	v_j(r_{j,\lambda})=\lambda.
\end{equation}
The regions $r<r_{j,\lambda}$ and $r>r_{j,\lambda}$ are respectively
classically forbidden and classically accessible at the energy
$\lambda$. We shall further assume the monotonicity condition
\begin{equation}\label{eq:radial-monotonicity-intro}
	\bigl(rv_j'(r)\bigr)'\geq0,
	\qquad r>r_{j,\lambda},
	\qquad j=1,2.
\end{equation}
In terms of the radial force $F_j(r)=-v_j'(r)>0$, this means that
$rF_j(r)$ is non-increasing in the classically accessible region.
Notice also that the virial condition is automatic in this region:
\begin{equation}\label{eq:radial-virial-intro}
	2\bigl(\lambda-v_j(r)\bigr)-rv_j'(r)>0,
	\qquad r>r_{j,\lambda}.
\end{equation}

\mn
Since the potentials are radial, the scattering amplitude is invariant
under simultaneous rotations of the incoming and outgoing directions.
More precisely, for every $R\in SO(n)$, the scattering amplitudes satisfy 
\begin{equation}\label{eq:radial-rotation-invariance}
	f_j(R\omega\to R\theta;\lambda,h)
	=
	f_j(\omega\to\theta;\lambda,h),
	\qquad j=1,2.
\end{equation}
The corresponding semiclassical differential cross section is defined by
\begin{equation}\label{eq:semiclassical-cross-section}
	\frac{d\sigma_{j,h}}{d\Omega}
	(\omega,\theta;\lambda)
	=
	|f_j(\omega\to\theta;\lambda,h)|^2.
\end{equation}
By rotational invariance, both the scattering amplitude and the
differential cross section depend only on the scattering angle between
$\omega$ and $\theta$. We may therefore fix an incoming direction
$\omega\in\mathbb S^{n-1}$ without loss of information.

\mn
Our second result shows that, in the radial setting, the differential
cross section alone is sufficient to recover the potential in the
classically accessible region.

\begin{theorem}\label{thm:radial}
	Let $V_j(x)=v_j(|x|)$, $j=1,2$, satisfy the short-range condition
	\eqref{eq:short-range} and assumptions
	\eqref{eq:radial-assumptions-intro}--\eqref{eq:radial-monotonicity-intro}.
	Fix an incoming direction $\omega\in\mathbb S^{n-1}$ and assume that
	\begin{equation}\label{eq:cross-section-equality-intro}
		\frac{d\sigma_{1,h}}{d\Omega}
		(\omega,\theta;\lambda)
		-
		\frac{d\sigma_{2,h}}{d\Omega}
		(\omega,\theta;\lambda)
		\longrightarrow0,
		\qquad h\longrightarrow0,
	\end{equation}
	for every $\theta\in\mathbb S^{n-1}\setminus\{\omega,-\omega\}$.
	Then
	\begin{equation}\label{eq:r-lambda-equality-intro}
		r_{1,\lambda}=r_{2,\lambda}=:r_\lambda,
	\end{equation}
	and
	\begin{equation}\label{eq:radial-uniqueness-intro}
		V_1(x)=V_2(x),
		\qquad |x|\geq r_\lambda.
	\end{equation}
\end{theorem}

\mn
Let us explain the main ingredients of the proof. The starting point is
the fixed-energy semiclassical asymptotic formula of Robert and Tamura
\cite{RobertTamura}. For a regular pair of incoming and outgoing
directions, the scattering amplitude has the form
\begin{equation}\label{eq:RT-asymptotics-intro}
	f(\omega\to\theta;\lambda,h)
	=
\sum_{j=1}^{\ell(\theta,\lambda)}
	\sigma(w_j;\lambda)^{-1/2}
	\exp\left(
	\frac{\ii}{h}\mathcal A_j(\omega,\theta;\lambda)
	-\frac{\ii\pi}{2}\mu_j
	\right)
	\bigl(1+\mathcal O(h)\bigr),
\end{equation}
where the sum runs over the classical scattering trajectories connecting
the incoming direction $\omega$ to the outgoing direction $\theta$.
Here $\sigma(w_j;\lambda)$ is the corresponding angular density,
$\mathcal A_j$ the classical action, and $\mu_j$ the Maslov index.
The notions of regular direction and classical branch, as well as the
precise definitions of the quantities entering
\eqref{eq:RT-asymptotics-intro}, are recalled in
Section~\ref{sec:radial}.

\mn
For general potentials, several classical trajectories may connect the same
incoming and outgoing directions. The corresponding contributions in
\eqref{eq:RT-asymptotics-intro} may then interfere, so that the semiclassical
differential cross section does not directly reduce to the classical
differential cross section of a single trajectory. In the radial setting considered here,
however, the classical dynamics can be described in terms of the
deflection function
$
b\longmapsto\Theta_\lambda(b),
$
which associates with each impact parameter $b>0$ the corresponding
scattering angle. Its precise definition is given in
Section~\ref{sec:radial}. Under
condition~\eqref{eq:radial-monotonicity-intro}, we shall prove that
$\Theta_\lambda$ is strictly decreasing from $\pi$ to $0$.

\mn 
For incoming and outgoing directions $\omega,\theta\in\mathbb S^{n-1}$,
we denote their scattering angle by
\begin{equation}\label{eq:scattering-angle-intro}
	\Theta
	=
	\arccos(\omega\cdot\theta)\in[0,\pi].
\end{equation}
Thus a trajectory with impact parameter $b$ and outgoing direction
$\theta$ satisfies
\[
\Theta=\Theta_\lambda(b).
\]
The strict monotonicity of $\Theta_\lambda$ therefore implies that,
for every $\Theta\in(0,\pi)$, there is a unique impact parameter $b$.
For fixed incoming and outgoing directions $\omega$ and $\theta$
with this scattering angle, the direction of the impact vector is
also uniquely determined, yielding a unique classical trajectory.
Moreover, as shown in Section~\ref{sec:radial}, the corresponding
outgoing directions are regular. The Robert--Tamura formula thus applies with
$\ell(\theta,\lambda)=1$. In particular, there are no interference
terms, and since $\sigma(w;\lambda)^{-1}$ is precisely the classical
differential cross section associated with the unique trajectory, we
obtain
\begin{equation}\label{eq:RT-cross-section-intro}
	|f(\omega\to\theta;\lambda,h)|^2
	\longrightarrow
	\frac{d\sigma_{\rm cl}}{d\Omega}(\Theta;\lambda),
	\qquad
	0<\Theta<\pi.
\end{equation}

\mn
The classical differential cross section determines the inverse function \(b=b(\Theta)\) of the deflection function \(b\mapsto\Theta_\lambda(b)\), and hence determines the deflection function itself. A classical Firsov--Abel inversion
\cite{Firsov1953,KellerKayShmoys1956} then recovers the radial potential
throughout the classically accessible region. This proves
Theorem~\ref{thm:radial}.

\mn
We also show that the same strategy extends to nontrapping
compactly supported metric perturbations of the Euclidean Laplacian on $\R^n$, $ n \geq 3$. 
The same Gaussian-state argument shows that the classical scattering
relation is determined under the threshold condition introduced above.
Since the metrics are Euclidean near the
boundary of a sufficiently large ball, a first variation argument
also gives the lengths of the corresponding geodesics.
Combining these lens data with the rigidity theorem of
Stefanov--Uhlmann--Vasy \cite[Theorem 8.1]{SUV2021}, we obtain
uniqueness of the metric up to a compactly supported
changes of coordinates, 
assuming that one of the metrics admits a smooth strictly convex
function whose zero set is the boundary of the ball.
No simplicity assumption is needed.  The precise statement is given in Theorem~\ref{thm:metric-simple-uniqueness}. 
Related scattering rigidity results for simple Riemannian metrics,
including metrics in a fixed conformal class and real-analytic
metrics, were obtained by Mu\~noz-Thon in the more general framework
of magnetic-potential systems \cite{MunozThon2024}.

\mn
Let us conclude by mentioning a possible extension of the radial result
to short-range potentials which are radial only near infinity.
At sufficiently large energy, classical inverse scattering results of
Jollivet \cite{JollivetRadial} provide uniqueness from suitable classical
scattering data. Extending the semiclassical recovery of the classical
scattering data to this setting would therefore lead to a corresponding
semiclassical inverse result. We do not pursue this question here.

\mn
The paper is organized as follows. In Section~\ref{sec:compact-support-recovery},
we recover the classical scattering map from the semiclassical scattering
matrix for compactly supported potentials, using Ingremeau's propagation
of Gaussian states, and then deduce uniqueness of the potential from the
rigidity theorem of Stefanov--Uhlmann--Vasy. Section~3 is devoted to
radial short-range potentials. After studying the classical radial
dynamics, we use the Robert--Tamura asymptotics and the Firsov--Abel
inversion formula to recover the potential in the classically accessible
region from the semiclassical differential cross section. Finally, Section~\ref{sec:metric-perturbations} extends the argument
to nontrapping compactly supported perturbations of the Euclidean
metric under a strict convexity assumption, without requiring
simplicity.

\section{Compactly supported potentials}
\label{sec:compact-support-recovery}

\subsection{Recovery of the scattering map }

\mn
By a simple scaling argument, we may assume without loss of generality
that the fixed non-trapping energy is $\lambda=1/2$. We therefore recall
Ingremeau's result \cite{Ingremeau} in this normalized setting, in the
form needed below. For
\begin{equation}\label{eq:Ing-rho-fixed}
	\rho=(\omega,z)\in T^*\mathbb S^{n-1},
	\qquad z\in\omega^\perp,
\end{equation}
we denote by $\kappa=\kappa_{1/2}$ the classical scattering map
introduced in \eqref{eq:classical-scattering-map}. Thus, if
\begin{equation}\label{eq:Ing-kappa-rho-compact}
	\rho_1=(\omega_1,z_1)=\kappa(\rho),
\end{equation}
then $\rho_1$ is the outgoing phase-space point associated with the
incoming datum $\rho$.

\mn
Following Ingremeau \cite{Ingremeau}, let $\Gamma_0$ be a symmetric
complex matrix with positive definite real part, let $Q_0$ be a
polynomial, and choose $\chi\in C^\infty(\mathbb R;[0,1])$ such that
\[
\chi(r)=1\quad\text{for }r\leq\frac12,
\qquad
\chi(r)=0\quad\text{for }r\geq\frac34.
\]
Associated with $\rho=(\omega,z)\in T^*\mathbb S^{n-1}$, we define
the Gaussian state
\begin{equation}\label{eq:Ing-Gaussian-compact}
	\varphi_{\omega, z,\Gamma_0,Q_0}(\widehat x;h)
	=
	\chi\left(
	\frac{|\widehat x-\omega|}{h^{1/3}}
	\right)
	Q_0\left(
	\frac{\widehat x-\omega}{\sqrt h}
	\right)
	e^{-\frac{\ii}{h}z\cdot\widehat x}
	e^{-\frac{1}{2h}
		(\widehat x-\omega)\cdot
		\Gamma_0(\widehat x-\omega)},
\end{equation}
where $\widehat x\in\mathbb S^{n-1}$. The Gaussian factor localizes the state at the scale $\sqrt h$ around
$\omega$, while the oscillatory factor localizes it microlocally at
the cotangent variable $z$. Thus
$\varphi_{\rho,\Gamma_0,Q_0}$ is microlocally concentrated near the
phase-space point $\rho=(\omega,z)$.

\mn
The Gaussian states used by Ingremeau may be viewed as the scattering
counterpart of the coherent states familiar from semiclassical
propagation. In particular, their localization in phase space and
their propagation along the classical scattering map are analogous
to the semiclassical propagation of Gaussian coherent states; see,
for instance, Combescure and Robert \cite{CombescureRobert}.

\mn
We now recall the main result of \cite{Ingremeau}, which describes the action of the scattering matrix on these Gaussian states at the non-trapping energy $\lambda=1/2$. With the conventions adopted in Section 1, Ingremeau's scattering matrix coincides with \(S(\lambda,h)\), $\lambda = \frac12$. Let
\begin{equation}\label{eq:Ing-outgoing-point}
	(\omega_1,z_1)
	=
	\kappa(\omega,z).
\end{equation}
Then there exist a real number $\delta_1$, a symmetric complex matrix
$\Gamma_1$ with positive definite real part, and polynomials
$Q_1^k$, $k\in\mathbb N$, such that the following holds. For every
$N\in\mathbb N$, setting
\begin{equation}\label{eq:Ing-QN}
	\widetilde Q_1^N
	=
	\sum_{k=0}^N h^{k/2}Q_1^k,
\end{equation}
one has
\begin{equation}\label{eq:Ing-main}
	S(\lambda,h)\varphi_{\omega,z,\Gamma_0,Q_0}
	=
	e^{i\delta_1/h}
	\varphi_{\omega_1,z_1,\Gamma_1,\widetilde Q_1^N}
	+
	R_N,
\end{equation}
where
\begin{equation}\label{eq:Ing-remainder}
	\|R_N\|_{C^0(\mathbb S^{n-1})}
	=
	O\left(h^{(N+1)/2}\right).
\end{equation}
Thus, to arbitrary order in the semiclassical expansion, the
scattering matrix maps a Gaussian state localized near the incoming
phase point $(\omega,z)$ to a Gaussian state localized near the
outgoing phase point $\kappa(\omega,z)$. This direct transport of
phase-space centers is the feature of Ingremeau's approach that will
be used below to recover the classical scattering map from the
semiclassical scattering matrix.

\mn
For the initial Gaussian state in the proof below, we choose
$\Gamma_0=I$ and $Q_0=1$.  We therefore set for $\rho=(\omega,z)$,
\begin{equation}\label{eq:Ing-standard-Gaussian}
	\phi_{\rho,h}(\widehat x)
	=
	c_h\,
	\chi\left(
	\frac{|\widehat x-\omega|}{h^{1/3}}
	\right)
	e^{-\frac{\ii}{h}z\cdot\widehat x}
	e^{-\frac{|\widehat x-\omega|^2}{2h}},
\end{equation}
where $c_h>0$ is chosen so that
\begin{equation}\label{eq:Ing-standard-normalization}
	\|\phi_{\rho,h}\|_{L^2(\mathbb S^{n-1})}=1.
\end{equation}

\mn
Alexandrova's Fourier integral description shows that the semiclassical
scattering matrix microlocally encodes the classical scattering relation.
Here, using Ingremeau's propagation result, we show that the quantitative
condition
\begin{equation}
\liminf_{h\to0}
\|S_1(\lambda,h)-S_2(\lambda,h)\|_
{\mathcal B(L^2(\mathbb S^{n-1}))}
<2
\end{equation}
is sufficient to ensure that the corresponding classical scattering
maps coincide. We state the result for an arbitrary fixed energy $\lambda>0$,
since by a simple scaling one can reduce to the case $\lambda=1/2$.

\begin{proposition}\label{lem:recovery-kappa}
	Let $V_1,V_2\in C_c^\infty(\mathbb R^n)$ and assume that
	$\lambda>0$ is a non-trapping energy for both potentials.
	Let $S_j(\lambda,h)$ and $\kappa_{j,\lambda}$ denote respectively the
	semiclassical scattering matrix and the classical scattering map
	associated with $V_j$, $j=1,2$.
	If
	\begin{equation}\label{eq:scattering-matrices-close}
		\liminf_{h\to0}
		\|S_1(\lambda,h)-S_2(\lambda,h)\|_
		{\mathcal B(L^2(\mathbb S^{n-1}))}
		<2,
	\end{equation}
	then
	\begin{equation}\label{eq:scattering-maps-equal}
		\kappa_{1,\lambda}=\kappa_{2,\lambda}.
	\end{equation}
\end{proposition}

\begin{proof}
The proof is elementary but uses a nice trick adapted from
Halmos \cite[Solution 94, p.~235]{Halmos1982}. We argue by contraposition. Suppose that
$\kappa_{1,\lambda}\neq\kappa_{2,\lambda}$, and set
\begin{equation}\label{eq:relative-scattering}
	U_h=S_1(\lambda,h)^{-1}S_2(\lambda,h),
	\qquad
	\kappa=\kappa_{1,\lambda}^{-1}\circ\kappa_{2,\lambda}.
\end{equation}
Throughout the proof, operator norms are taken on
$L^2(\mathbb S^{n-1})$.
Since the scattering matrices are unitary, $U_h$ is unitary and
\begin{equation}\label{eq:relative-scattering-norm}
	\|S_1(\lambda,h)-S_2(\lambda,h)\|=\|I-U_h\|.
\end{equation}
We shall prove that \footnote{In fact, the lower bound is sharp: the liminf equals \(2\).}
\begin{equation}\label{eq:relative-scattering-target}
	\liminf_{h\to0}\|I-U_h\|\geq2.
\end{equation}
	
	\medskip
	\noindent
	\textit{Step 1: Construction of a finite orbit segment.}
	The map $\kappa$ is a nontrivial smooth diffeomorphism of the
	connected manifold $T^*\mathbb S^{n-1}$.
Since both potentials have compact support, the scattering maps
are the identity for large impact parameters. Thus $\kappa$
is the identity outside a compact set.
Suppose now that $\kappa$ has finite order, that is,
$\kappa^p=\mathrm{id}$ for some integer $p\geq1$.
Since $\kappa\neq\mathrm{id}$, the group $G=\langle\kappa\rangle$
is finite and nontrivial. It acts effectively on the connected
manifold $T^*\mathbb S^{n-1}$.
Since $\kappa$ is the identity on a nonempty open set,
every element of $G$ fixes this open set.
By \cite[Chapter III, Theorem 9.5]{Bredon1972},
this is impossible unless $G$ is trivial.
Thus $\kappa=\mathrm{id}$, contrary to our assumption.
	
\mn
Fix an integer $N\geq1$. By the above argument,
$\kappa^{N!}\neq\mathrm{id}$. Thus there exists
$\rho_0\in T^*\mathbb S^{n-1}$ such that
$
	\kappa^{N!}(\rho_0)\neq\rho_0.
$
We claim that the points
\begin{equation}\label{eq:kappa-orbit-segment}
	\rho_j=\kappa^j(\rho_0),\qquad 0\leq j\leq N,
\end{equation}
are pairwise distinct. Indeed, if $\rho_i=\rho_j$ for
$0\leq i<j\leq N$, then $\kappa^{j-i}(\rho_0)=\rho_0$.
Since $j-i$ divides $N!$, we obtain
$\kappa^{N!}(\rho_0)=\rho_0$, which is not possible.

	\medskip
	\noindent
	\textit{Step 2: Propagation of a Gaussian state.} Let $\phi_{\rho_0,h}$ be the normalized Gaussian state defined
	in \eqref{eq:Ing-standard-Gaussian}, and set
	\begin{equation}\label{eq:iterated-Gaussian-states}
		u_{j,h}=U_h^j\phi_{\rho_0,h},
		\qquad 0\leq j\leq N.
	\end{equation}
	Since $U_h$ is unitary, we have
	\begin{equation}\label{eq:iterated-Gaussian-norms}
		\|u_{j,h}\|_{L^2}=1,
		\qquad 0\leq j\leq N.
	\end{equation}
Here and below, $o(1)$ is understood as $h\to0$, with $N$ fixed. 
Ingremeau's result also holds for the inverse scattering matrix.
Thus $S_2(\lambda,h)$ maps a Gaussian state centered at $\rho$
to a Gaussian state centered at $\kappa_{2,\lambda}(\rho)$.
Similarly, $S_1(\lambda,h)^{-1}$ maps a Gaussian state centered
at $\rho$ to a Gaussian state centered at
$\kappa_{1,\lambda}^{-1}(\rho)$.
We take enough terms in Ingremeau's expansion so that the
$L^2$ error is still $o(1)$ for our normalized Gaussian states. By applying these results $j$ times, we see that $u_{j,h}$
is concentrated near $\rho_j=\kappa^j(\rho_0)$.

\mn
We claim that
\begin{equation}\label{eq:iterated-Gaussian-orthogonality}
	\langle u_{i,h},u_{j,h}\rangle_{L^2}=o(1),
	\qquad 0\leq i\neq j\leq N.
\end{equation}
Indeed, each $u_{j,h}$ is, up to an $o(1)$ error in $L^2$,
a Gaussian state centered at $\rho_j=(\omega_j,z_j)$,
and these points are all distinct.
There are two cases.
If $\omega_i\neq\omega_j$, the cutoffs have disjoint
supports for small $h$, since they are supported in balls
of radius $O(h^{1/3})$ around distinct points.
If $\omega_i=\omega_j=:\omega_*$, then $z_i\neq z_j$.
The scalar product of the Gaussian states contains
the oscillatory factor
\begin{equation}\label{eq:orbit-impact-phase}
	\exp\left(\frac{i}{h}(z_i-z_j)\cdot\widehat x\right).
\end{equation}
Since $z_i-z_j\in\omega_*^\perp$ is nonzero, the phase
has no critical point on the support for small $h$.
By non-stationary phase, the scalar product is $o(1)$.
Indeed, each integration by parts gives a factor $h$,
while each derivative of the Gaussian amplitude costs
at most $h^{-1/2}$.
This proves \eqref{eq:iterated-Gaussian-orthogonality}.

\medskip
\noindent
\textit{Step 3: Estimate of the operator norm.}
We use a simple construction similar to the one in Halmos's book
\cite[Solution 94, p.~235]{Halmos1982}.
We define
\begin{equation}\label{eq:alternating-sum}
	v_h=\sum_{j=0}^{N-1}(-1)^j u_{j,h}.
\end{equation}
Since the states $u_{j,h}$ are normalized and asymptotically
orthogonal by \eqref{eq:iterated-Gaussian-orthogonality},
we easily obtain
\begin{equation}\label{eq:alternating-sum-norm}
	\|v_h\|_{L^2}^2=N+o(1).
\end{equation}
Since $U_hu_{j,h}=u_{j+1,h}$, we note that
\begin{equation}\label{eq:alternating-sum-cancellation}
	(I+U_h)v_h=u_{0,h}+(-1)^{N-1}u_{N,h}.
\end{equation}
Using again \eqref{eq:iterated-Gaussian-orthogonality}, we see that
\begin{equation}\label{eq:alternating-sum-error}
	\|(I+U_h)v_h\|_{L^2}^2=2+o(1).
\end{equation}
By the parallelogram identity and since $U_h$ is unitary,
we have
\begin{equation}\label{eq:unitary-norm-identity}
	\|(I-U_h)v_h\|_{L^2}^2
	+\|(I+U_h)v_h\|_{L^2}^2
	=4\|v_h\|_{L^2}^2.
\end{equation}
It follows that
\begin{equation}\label{eq:relative-scattering-lower-bound}
	\|I-U_h\|^2
	\geq
	\frac{\|(I-U_h)v_h\|_{L^2}^2}{\|v_h\|_{L^2}^2}
	=4-\frac{2}{N}+o(1).
\end{equation}
First letting $h\to0$, with $N$ fixed, and then letting
$N\to\infty$, we obtain
\begin{equation}\label{eq:relative-scattering-liminf}
	\liminf_{h\to0}\|I-U_h\|\geq2.
\end{equation}
This proves the proposition.
\end{proof}

\subsection{From the scattering map to the potential}
\label{subsec:scattering-map-potential}

We now turn to the classical inverse problem. We assume first that the
energy lies above the potentials,
\begin{equation}\label{eq:energy-above-potential}
	\lambda>\sup_{\mathbb R^n}V_j,
	\qquad j=1,2,
\end{equation}
and that the virial condition
\begin{equation}\label{eq:virial-condition}
	2(\lambda-V_j(x))-x\cdot\nabla V_j(x)>0,
	\qquad x\in\mathbb R^n,\qquad j=1,2,
\end{equation}
holds. The condition \eqref{eq:virial-condition} implies that the
Hamiltonian flow at energy $\lambda$ is non-trapping.

\mn
By the Maupertuis principle, under
\eqref{eq:energy-above-potential}, the trajectories at energy
$\lambda$, up to a reparametrization, are the geodesics of the Jacobi
metrics
\begin{equation}\label{eq:Jacobi-metrics}
	g_j=2(\lambda-V_j)g_0,
	\qquad j=1,2,
\end{equation}
where $g_0$ denotes the Euclidean metric on $\mathbb R^n$. Equivalently,
we may write
\begin{equation}\label{eq:Jacobi-speeds}
	g_j=c_j^{-2}g_0,
	\qquad
	c_j=\bigl(2(\lambda-V_j)\bigr)^{-1/2},
	\qquad j=1,2.
\end{equation}

\mn
This reduction was already used by Alexandrova \cite{Alexandrova}
under the  assumption that the Jacobi metric is simple.
In that setting, she used an argument implicit in Michel
\cite{Michel1981} to recover the boundary distance function from the
scattering relation, and then applied the boundary rigidity results of
Mukhometov \cite{Mukhometov1981} and Mukhometov--Romanov
\cite{MukhometovRomanov1978} to determine the conformal factor.
More recently, in a purely classical setting, Mu\~noz-Thon
\cite{MunozThon2024} obtained scattering rigidity results at a single
energy for simple magnetic-potential systems. In particular, in the non-magnetic case, his result recovers the
classical rigidity statement used by Alexandrova: the classical
scattering relation determines the potential when the associated
Jacobi metric is simple.

\mn
The geometric assumption used here is different. We do not require
the Jacobi metrics to be simple, but instead use the rigidity theorem
of Stefanov--Uhlmann--Vasy \cite{SUV2016}, which applies when the
manifold admits a strictly convex foliation. Our virial condition
provides precisely such a foliation for the Jacobi metrics. Since this
condition does not exclude conjugate points, it applies to geometries
which need not be simple. We note that Mu\~noz-Thon
\cite{MunozThon2024} explicitly raised the question of extending
scattering rigidity for magnetic-potential systems from the simple
setting to the strictly convex foliation framework.

\begin{definition}\label{def:convex-foliation}
	Let $(M,g)$ be a compact Riemannian manifold with boundary.
	We say that $M$ admits a strictly convex foliation if there exist
	$T>0$ and a smooth function
	\[
	\rho:M\longrightarrow [0,+\infty)
	\]
	such that, for every $t\in[0,T)$, the level set
	\[
	\Sigma_t=\rho^{-1}(t)
	\]
	is a strictly convex hypersurface, $d\rho\neq0$ on $\Sigma_t$,
	$\Sigma_0=\partial M$, and
	\[
	M\setminus\bigcup_{0\leq t<T}\Sigma_t
	\]
	has empty interior.
\end{definition}

\mn
In our setting, the virial condition provides precisely such a
foliation. Indeed, on $M=\overline{B(0,R)}$, we take
$
	\rho(x)=R^2-|x|^2.
$
Then, for $0\leq t<R^2$,
\begin{equation}\label{eq:sigma-convex-foliation}
	\Sigma_t=\rho^{-1}(t)
	=
	\{|x|=\sqrt{R^2-t}\}.
\end{equation}
Thus the leaves of the foliation are precisely the Euclidean spheres
centered at the origin. 

\mn
The following lemma shows that the virial condition indeed guarantees
the strict convexity of these leaves with respect to the Jacobi metrics.
This observation is likely known, but we have been unable to find an
explicit statement of it in the literature.

\begin{lemma}\label{lem:virial-convex-foliation}
	Under \eqref{eq:energy-above-potential} and
	\eqref{eq:virial-condition}, the Euclidean spheres form a strictly
	convex foliation for the Jacobi metrics \eqref{eq:Jacobi-metrics}.
\end{lemma}

\begin{proof}
	It is enough to consider one potential $V$. Writing the Jacobi
	metric in the form
	\begin{equation}\label{eq:Jacobi-speed}
		g=c^{-2}g_0,
		\qquad
		c=\bigl(2(\lambda-V)\bigr)^{-1/2},
	\end{equation}
	we have
	\begin{equation}\label{eq:Herglotz-virial-equivalence}
		\frac{\partial}{\partial r}
		\left(\frac{r}{c}\right)
		=
		\frac{
			2(\lambda-V(x))-x\cdot\nabla V(x)
		}{
			\sqrt{2(\lambda-V(x))}
		}.
	\end{equation}
	Hence \eqref{eq:virial-condition} implies the
	Herglotz--Wiechert--Zoeppritz condition
	\begin{equation}\label{eq:Herglotz-condition}
		\frac{\partial}{\partial r}
		\left(\frac{r}{c}\right)>0.
	\end{equation}
	By \cite[Proposition~6.1]{SUV2016}, this is equivalent to the
	strict convexity of the Euclidean spheres for the Jacobi metric.
	As the concentric spheres foliate $\mathbb R^n\setminus\{0\}$,
	the conclusion follows.
\end{proof}

\mn
We briefly recall the notion of scattering relation used in
Stefanov--Uhlmann--Vasy \cite{SUV2016}. Let $(M,g)$ be a compact
Riemannian manifold with boundary. In our application below, $M$ will
simply be the closed Euclidean ball $\overline{B(0,R)}\subset\mathbb R^n$,
equipped with one of the Jacobi metrics \eqref{eq:Jacobi-metrics}. We denote by $\partial_-SM$ and
$\partial_+SM$ the sets of unit tangent vectors based at $\partial M$
and pointing respectively into and out of $M$. For $(x_-,v_-)\in\partial_-SM$, let $\gamma_{x_-,v_-}$ be the
unique unit speed geodesic issued from $(x_-,v_-)$. In our setting, the non-trapping assumption ensures that this geodesic
has a finite exit time, defining an outgoing point and direction
$(x_+,v_+)\in\partial_+SM$. The scattering relation is defined by
\begin{equation}\label{eq:SUV-scattering-relation}
	L(x_-,v_-)
	=
	(x_+,v_+).
\end{equation}


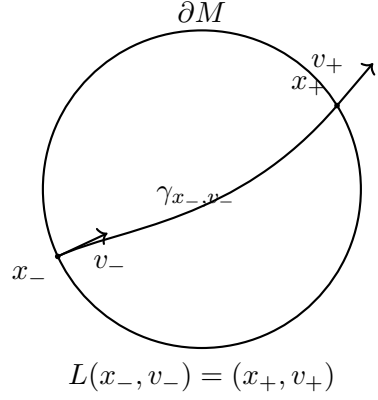
\begin{figure}[ht]
	\centering
	\begin{tikzpicture}[scale=0.7]
		\draw[thick] (0,0) circle (3);
		
		\coordinate (A) at (-2.72,-1.27);
		\coordinate (B) at (2.55,1.58);
		
		\fill (A) circle (1.5pt);
		\fill (B) circle (1.5pt);
		
		\draw[thick]
		(A)
		.. controls (-1.55,-0.72) and (0.65,-0.65) ..
		(B);
		
		\draw[->,thick]
		(A) -- ++(0.94,0.44)
		node[midway,below right] {$v_-$};
		
		\draw[->,thick]
		(B) -- ++(0.68,0.80)
		node[midway,above left] {$v_+$};
		
		\node[below left] at (A) {$x_-$};
		\node[above left] at (B) {$x_+$};
		\node at (0,3.35) {$\partial M$};
		\node at (-0.1,-0.15) {$\gamma_{x_-,v_-}$};
		
		\node at (0,-3.55)
		{$L(x_-,v_-)=(x_+,v_+)$};
	\end{tikzpicture}
	\caption{The scattering relation of a Riemannian metric:
		an incoming point and direction $(x_-,v_-)$ are mapped to the
		corresponding exit point and direction $(x_+,v_+)$.}
	\label{fig:SUV-scattering-relation}
\end{figure}


\mn
By Proposition~\ref{lem:recovery-kappa}, the assumption
\begin{equation}\label{eq:scattering-matrices-close-inverse}
	\liminf_{h\to0}
	\|S_1(\lambda,h)-S_2(\lambda,h)\|_
	{\mathcal B(L^2(\mathbb S^{n-1}))}
	<2,
\end{equation}
implies
\begin{equation}\label{eq:classical-scattering-equality}
	\kappa_{1,\lambda}
	=
	\kappa_{2,\lambda}.
\end{equation}
Since $V_1$ and $V_2$ are compactly supported, we may choose
$R>0$ such that
\begin{equation}\label{eq:supports-in-ball}
	\operatorname{supp}V_1\cup\operatorname{supp}V_2
	\subset B(0,R).
\end{equation}
Outside $B(0,R)$ the Hamiltonian trajectories are straight lines.
Hence the incoming and outgoing asymptotic data defining
$\kappa_{j,\lambda}$ determine uniquely the corresponding entry and
exit points and directions on $\partial B(0,R)$. It follows from \eqref{eq:classical-scattering-equality} that the
scattering relations $L_j$ of the Jacobi metrics on $B(0,R)$ satisfy
\begin{equation}\label{eq:boundary-scattering-relations}
	L_1=L_2.
\end{equation}
Notice that no travel-time information is needed here, since the
rigidity result \cite[Theorem~1.4]{SUV2016} only requires equality of
the scattering relations. Finally, since $V_1=V_2=0$ near
$\partial B(0,R)$, the Jacobi metrics coincide there:
\begin{equation}\label{eq:Jacobi-boundary-equality}
	g_1=g_2=2\lambda g_0
	\qquad\text{near }\partial B(0,R).
\end{equation}

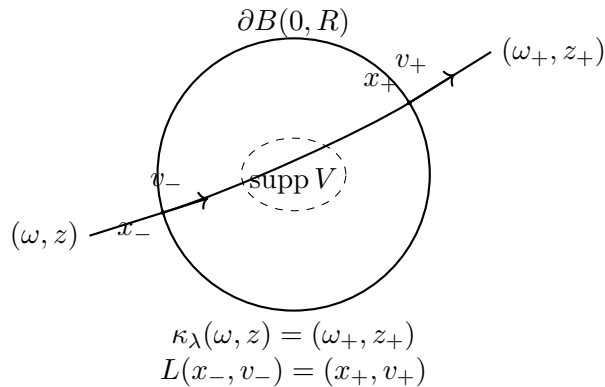
\begin{figure}[ht]
	\centering
	\begin{tikzpicture}[scale=0.6]
		\draw[thick] (0,0) circle (3);
		
		\draw[dashed] (0,0) ellipse (1.15 and 0.8);
		\node at (0,-0.15) {$\operatorname{supp}V$};
		
		\coordinate (A) at (-2.88,-0.84);
		\coordinate (B) at (2.55,1.58);
		
		\draw[thick] (-4.5,-1.35) -- (A);
		\node[left] at (-4.5,-1.35) {$(\omega,z)$};
		
		\draw[thick]
		(A)
		.. controls (-1.30,-0.35) and (1.55,0.96) ..
		(B);
		
		\draw[thick] (B) -- (4.35,2.70);
		\node[right] at (4.35,2.70) {$(\omega_+,z_+)$};
		
		\fill (A) circle (1.5pt);
		\fill (B) circle (1.5pt);
		
		\node[below left] at (A) {$x_-$};
		\node[above left] at (B) {$x_+$};
		
		\draw[thick,->]
		(A) -- ++(1.00,0.315)
		node[pos=0.68,above left] {$v_-$};
		
		\draw[thick,->]
		(B) -- ++(1.00,0.622)
		node[pos=0.65,above left] {$v_+$};
		
		\node at (0,3.35) {$\partial B(0,R)$};
		
	\node at (0,-3.55)
	{$\kappa_\lambda(\omega,z)=(\omega_+,z_+)$};
	
	\node at (0,-4.35)
	{$L(x_-,v_-)=(x_+,v_+)$};
	\end{tikzpicture}
	\caption{Relation between the classical scattering map
		$\kappa_\lambda$ and the scattering relation $L$ on
		$\partial B(0,R)$. Outside the support of $V$, the trajectories
		are straight lines, so that the asymptotic data determine uniquely
		the corresponding entry and exit data on the boundary.}
	\label{fig:scattering-map-lens-relation}
\end{figure}

\mn
We may therefore apply the global rigidity theorem of
Stefanov--Uhlmann--Vasy \cite[Theorem~1.4]{SUV2016}. Indeed,
Lemma~\ref{lem:virial-convex-foliation} provides the required strictly
convex foliation. It follows that
$
	g_1
	=
	g_2.
$
In view of \eqref{eq:Jacobi-metrics}, we conclude that
$
	V_1
	=
	V_2.
$


\section{Radial short-range potentials}
\label{sec:radial}

We now consider smooth radial repulsive short-range potentials
\begin{equation}\label{eq:radial-short-range}
	V(x)=v(|x|),
	\qquad
	v\in C^\infty([0,+\infty)),
	\qquad
	|\partial_r^k v(r)|
	\leq C_k\langle r\rangle^{-\mu-k},
	\qquad \mu>1.
\end{equation}
Following the classical setting of Keller--Kay--Shmoys \cite{KellerKayShmoys1956},
we assume that
\begin{equation}\label{eq:radial-keller}
	v(r)>0,
	\qquad
	v'(r)<0,
	\qquad r>0,
\end{equation}
and that
\begin{equation}\label{eq:radial-energy}
	v(0)>\lambda.
\end{equation}

\mn
Since $v(r)\to0$ as $r\to+\infty$, assumptions
\eqref{eq:radial-keller}--\eqref{eq:radial-energy} imply that there exists
a unique $r_\lambda>0$ such that
\begin{equation}
v(r_\lambda)=\lambda.
\end{equation}
The region $0<r<r_\lambda$ is classically forbidden at energy
$\lambda$, whereas the classically accessible region is
$
r>r_\lambda.
$
Thus, \eqref{eq:radial-energy} ensures the presence of a nonempty
classically forbidden core. This assumption also plays an important role
in the scattering geometry. Indeed, if $v(0)<\lambda$, one can find that
the deflection angle $\Theta_\lambda(b)$, defined below in
\eqref{eq:deflection}, tends to zero both as the impact parameter
$b$ tends to zero and as $b$ tends to infinity. Hence
$\Theta_\lambda(b)$ cannot be strictly monotone in this case, and several
impact parameters may correspond to the same scattering angle. The
presence of the forbidden core, together with the monotonicity assumption
introduced below, will instead imply that $\Theta_\lambda$ is strictly
monotone.

\begin{figure}[ht]
	\centering
	\begin{tikzpicture}[scale=0.9]
		
		\draw[->] (0,0) -- (8.2,0) node[right] {$r$};
		\draw[->] (0,0) -- (0,5.2) node[above] {$v(r)$};
		
		\draw[dashed] (0,2.2) -- (7.6,2.2);
		\node[left] at (0,2.2) {$\lambda$};
		
		\draw[thick]
		(0.25,4.6)
		.. controls (1.0,4.2) and (1.5,3.0) ..
		(2.2,2.2)
		.. controls (3.2,1.2) and (5.2,0.45) ..
		(7.4,0.18);
		
		\draw[dashed] (2.2,0) -- (2.2,2.2);
		\node[below] at (2.2,0) {$r_\lambda$};
		
		\node at (1.05,1.15) {\small classically forbidden};
		\node at (1.05,0.70) {\small $v(r)>\lambda$};
		
		\node at (5.05,1.15) {\small classically accessible};
		\node at (5.05,0.70) {\small $v(r)<\lambda$};
		
	\end{tikzpicture}
	\caption{A radial repulsive potential in the setting of
		Keller--Kay--Shmoys. The unique radius $r_\lambda$ is defined by
		$v(r_\lambda)=\lambda$. At the fixed energy $\lambda$, the region
		$0<r<r_\lambda$ is classically forbidden, whereas $r>r_\lambda$
		is classically accessible.}
	\label{fig:radial-accessible-region}
\end{figure}
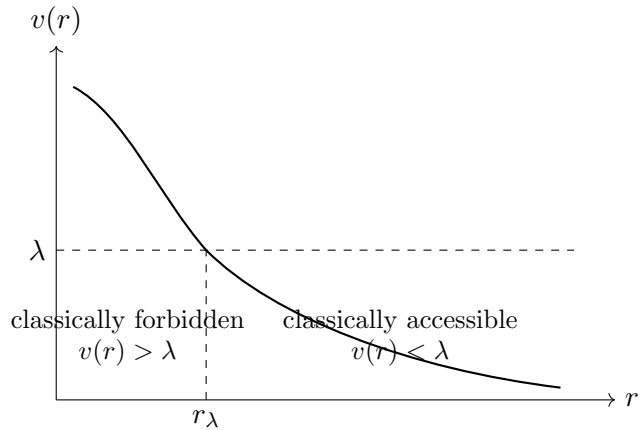

\mn 
In the radial setting, the virial condition
\eqref{eq:virial-condition} takes the form
\begin{equation}\label{eq:radial-virial}
	2\bigl(\lambda-v(r)\bigr)-r v'(r)>0.
\end{equation}
On the energy surface
\[
\frac12|\xi|^2+v(r)=\lambda,
\]
one necessarily has $v(r)\leq\lambda$, and hence $r\geq r_\lambda$.
Since $v'(r)<0$, \eqref{eq:radial-virial} is therefore satisfied on
the whole energy surface, including at $r=r_\lambda$. In particular,
the classical dynamics at energy $\lambda$ is non-trapping, so that
the semiclassical scattering results of Robert--Tamura apply.

\mn
We further assume that
\begin{equation}\label{eq:radial-convexity}
	\bigl(r v'(r)\bigr)'\geq0,
	\qquad r>r_\lambda.
\end{equation}
This condition will be used below to prove the strict monotonicity of the
deflection function $b\mapsto\Theta_\lambda(b)$, and hence the uniqueness
of the impact parameter associated with a given scattering angle. Related sufficient conditions for the monotonicity of the deflection
function were obtained, in the setting of compactly supported radial
potentials, by Pulvirenti--Saffirio--Simonella
\cite[Appendix A, condition (A.8)]{PulvirentiSaffirioSimonella2014}
and Gallagher--Saint-Raymond--Texier
\cite[Lemma 8.3.1]{GallagherSaintRaymondTexier2014}.
Their condition is
\[
rv''(r)+2v'(r)\geq0,
\]
whereas \eqref{eq:radial-convexity} reads
\[
rv''(r)+v'(r)\geq0.
\]
Since $v'<0$, their condition is stronger than
\eqref{eq:radial-convexity}. Moreover, the potentials considered here
are short-range and need not be compactly supported. Although the
argument below is similar in spirit to these previous monotonicity
arguments, their results do not apply directly in the present setting.
We therefore give the complete argument below.

\begin{remark}\label{rem:radial-convexity}
	Condition \eqref{eq:radial-convexity} has a natural interpretation
	in terms of the radial force
	\[
	F(r):=-v'(r)>0.
	\]
	It is equivalent to requiring that the weighted radial force
	$rF(r)=-rv'(r)$ be non-increasing on the classically accessible region.
	For instance, for the inverse-power profile
	\begin{equation}\label{eq:power-potential}
		v(r)=Cr^{-\rho},
		\qquad r\geq R>0,
		\qquad C>0,\quad \rho>0,
	\end{equation}
	one has
	\begin{equation}\label{eq:power-convexity}
		\bigl(rv'(r)\bigr)'
		=
		\rho^2 C r^{-\rho-1}>0.
	\end{equation}
\end{remark}


\subsection{Review of classical radial dynamics}
\label{subsec:radial-dynamics}

\mn
We briefly recall some standard facts about the classical dynamics of
radial potentials; see, for instance,
Goldstein~\cite{Goldstein},
Keller--Kay--Shmoys~\cite{KellerKayShmoys1956},
Jollivet~\cite{JollivetRadial},
Pulvirenti--Saffirio--Simonella~\cite{PulvirentiSaffirioSimonella2014},
and Gallagher--Saint-Raymond--Texier~\cite{GallagherSaintRaymondTexier2014}. These facts are classical, and we recall
them here only to fix the notation and sign conventions used below. 

\mn 
Under the assumptions
\eqref{eq:radial-keller}--\eqref{eq:radial-energy}, the motion at the
fixed energy $\lambda$ takes place in the classically accessible
region $r\geq r_\lambda$. Together with the additional assumption
\eqref{eq:radial-convexity}, these hypotheses will imply that each scattering angle in $(0,\pi)$ corresponds to a unique
impact parameter $b>0$.

\mn
Let $b=|z|>0$ be the impact parameter. Since the force is central,
the trajectory lies in the scattering plane
\begin{equation}\label{eq:scattering-plane}
	\Pi_{\omega,z}:=\operatorname{span}\{\omega,z\},
\end{equation}
and the angular momentum $q_\infty\wedge p_\infty$ is conserved;
see, for instance, \cite[Sections~3.2 and~3.10]{Goldstein}.
Its magnitude $L$ is determined by the incoming asymptotics:
\begin{equation}\label{eq:angular-momentum-b}
	L=|q_\infty\wedge p_\infty|
	=\sqrt{2\lambda}\,b.
\end{equation}


\mn
Let $(r(t),\varphi(t))$ be polar coordinates in the scattering plane
$\Pi_{\omega,z}$, with $r(t)=|q_\infty(t;z,\lambda)|$.
Conservation of angular momentum gives
\begin{equation}\label{eq:angular-momentum-polar}
	r(t)^2|\dot\varphi(t)|=L=\sqrt{2\lambda}\,b.
\end{equation}
Using conservation of energy, we obtain the radial equation
\begin{equation}\label{eq:radial-energy-equation}
	\frac12\dot r(t)^2
	+\frac{L^2}{2r(t)^2}
	+v(r(t))
	=
	\frac12\dot r(t)^2
	+\lambda\frac{b^2}{r(t)^2}
	+v(r(t))
	=\lambda.
\end{equation}
It is useful to introduce the effective radial potential
\begin{equation}\label{eq:effective-radial-potential}
	v_{\mathrm{eff}}(r;b)
	:=
	v(r)+\lambda\frac{b^2}{r^2}.
\end{equation}
Under the assumption $v'(r)<0$, one has
\begin{equation}\label{eq:effective-radial-potential-derivative}
	\partial_r v_{\mathrm{eff}}(r;b)
	=
	v'(r)-2\lambda\frac{b^2}{r^3}
	<0,
	\qquad r>0.
\end{equation}
Thus $v_{\mathrm{eff}}(\,\cdot\,;b)$ is strictly decreasing. Moreover,
for $b>0$,
\begin{equation}\label{eq:effective-radial-potential-limits}
	v_{\mathrm{eff}}(r;b)\longrightarrow+\infty
	\quad\text{as }r\to0^+,
	\qquad
	v_{\mathrm{eff}}(r;b)\longrightarrow0
	\quad\text{as }r\to+\infty.
\end{equation}
Consequently, for every $b>0$, there exists a unique turning point
$r_0(b)$ characterized by
\begin{equation}\label{eq:radial-turning-point}
	v(r_0(b))
	+
	\lambda\frac{b^2}{r_0(b)^2}
	=
	\lambda.
\end{equation}
Since \eqref{eq:radial-turning-point} yields
$v(r_0(b))<\lambda=v(r_\lambda)$ and $v$ is strictly decreasing, one
necessarily has $r_0(b)>r_\lambda$. Geometrically, $r_0(b)$ is the
distance of closest approach of the trajectory to the origin. More
precisely, if $t_0$ denotes the time at which the trajectory reaches
the turning point, then $\dot r(t_0)=0$ and
\begin{equation}\label{eq:radial-closest-approach}
	r_0(b)
	=
	r(t_0)
	=
	\min_{t\in\mathbb R} r(t).
\end{equation}
By rotational invariance, the distance of closest approach depends on
the impact vector $z$ only through $b=|z|$. Rewriting
\eqref{eq:radial-turning-point}, we obtain
\begin{equation}\label{eq:rho-turning}
	b
	=
	\rho_\lambda(r_0(b)),
	\qquad
	\rho_\lambda(r)
	:=
	r\sqrt{1-\frac{v(r)}{\lambda}},
	\qquad r\geq r_\lambda.
\end{equation}
Thus, the function $\rho_\lambda$ relates the radius $r_0(b)$ of the
turning point to the corresponding impact parameter $b$. It will play
a central role in the inversion argument below.

\mn
By \eqref{eq:radial-keller}, for $r>r_\lambda$,
\begin{equation}\label{eq:rho-derivative}
	\rho_\lambda'(r)
	=
	\sqrt{1-\frac{v(r)}{\lambda}}
	\left(
	1-\frac{r v'(r)}
	{2(\lambda-v(r))}
	\right)
	>0,
\end{equation}
since $\lambda-v(r)>0$ and $v'(r)<0$. Hence $\rho_\lambda$ is strictly
increasing from $0$ to $+\infty$ on $[r_\lambda,+\infty)$, and
\begin{equation}\label{eq:turning-point-parametrization}
	b=\rho_\lambda(r_0(b))
\end{equation}
provides a one-to-one parametrization of the radial turning points by
the impact parameter $b\geq0$. On the outgoing part of the trajectory, $r(t)$ increases from
$r_0(b)$ to $+\infty$. From \eqref{eq:radial-energy-equation}, we have
\begin{equation}\label{eq:radial-velocity}
	\dot r(t)
	=
	\sqrt{2\lambda}\,
	\sqrt{
		1-\frac{v(r(t))}{\lambda}
		-\frac{b^2}{r(t)^2}
	}.
\end{equation}
For $r>r_0(b)$, we may therefore use $r$ as a parameter along the
trajectory. Combining \eqref{eq:radial-velocity} with
\eqref{eq:angular-momentum-polar} and using the chain rule, we obtain
\begin{equation}\label{eq:angular-radial-relation}
	\left|\frac{d\varphi}{dr}\right|
	=
	\frac{|\dot\varphi(t)|}{\dot r(t)}
	=
	\frac{b}{
		r^2
		\sqrt{
			1-\dfrac{v(r)}{\lambda}
			-\dfrac{b^2}{r^2}
	}}.
\end{equation}
Then, the angular variation along the outgoing part of the trajectory is
\begin{equation}\label{eq:Phi-def}
	\Phi_\lambda(b)
	:=
	b\int_{r_0(b)}^{+\infty}
	\frac{dr}{
		r^2
		\sqrt{
			1-\dfrac{v(r)}{\lambda}
			-\dfrac{b^2}{r^2}
	}}.
\end{equation}
By the symmetry of the central-force trajectory with respect to the
line through the origin and the point of closest approach, the angular
variation from the incoming asymptote to the turning point equals that
from the turning point to the outgoing asymptote. Hence the total
angular variation between the two asymptotes is $2\Phi_\lambda(b)$.
We therefore define the signed deflection function by
\begin{equation}\label{eq:deflection}
	\Theta_\lambda(b)
	:=
	\pi-2\Phi_\lambda(b)
	=
	\pi
	-
	2b\int_{r_0(b)}^{+\infty}
	\frac{dr}{
		r^2
		\sqrt{
			1-\dfrac{v(r)}{\lambda}
			-\dfrac{b^2}{r^2}
	}}.
\end{equation}
This is the standard formula for the deflection angle in a central
potential; see Goldstein~\cite[p.~108, Eq.~(3.96)]{Goldstein}, with the
notation adapted to the present setting.
\begin{remark}\label{rem:free-deflection}
	Let us check the above convention in the free case $V=0$. Then
	$v(r)=0$ and the turning point equation gives
	$
		r_0(b)=b.
$
	Consequently,
	\begin{equation}\label{eq:free-Phi}
		\Phi_\lambda(b)
		=
		b\int_b^{+\infty}
		\frac{dr}{
			r^2\sqrt{1-\dfrac{b^2}{r^2}}
		}
		=
		\frac{\pi}{2}.
	\end{equation}
	Therefore,
	$
		\Theta_\lambda(b)
		=
		\pi-2\Phi_\lambda(b)
		=
		0,
	$
	as expected, since in the absence of a potential the incoming and
	outgoing directions coincide.
\end{remark}

\mn
The geometry of the resulting radial scattering trajectory is illustrated
in Figure~\ref{fig:radial-scattering}.

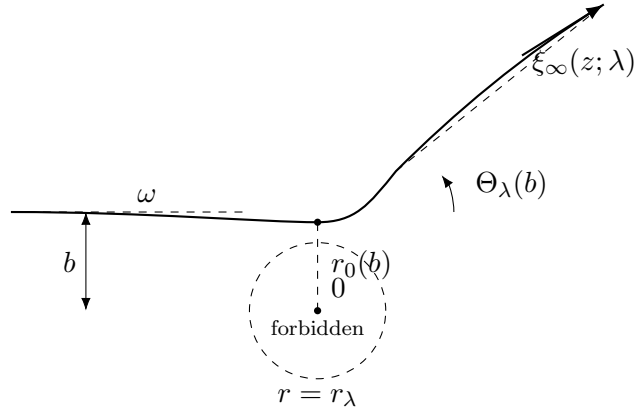
\begin{figure}[ht]
	\centering
\begin{tikzpicture}[scale=0.7,>=Latex]
	\pgfmathsetmacro{\rturn}{27/14}
	\pgfmathsetmacro{\rcore}{6*sqrt(2)/7}
	\coordinate (O) at (0,0);
	\coordinate (C) at (110:\rturn);
	
	\draw[dashed] (O) circle (\rcore);
	\fill (O) circle (1.3pt);
	\node[below right] at (O) {$0$};
	\node at (0,-0.55) {\scriptsize forbidden};
	\node[below] at (0,-\rcore) {$r=r_\lambda$};
	\draw[gray,densely dotted] (-4.7,0) -- (O);
	
	\draw[dashed,gray] (-5.5,1.5) -- (-0.05,1.5);
	\draw[dashed,gray]
	({-1.5*sin(40)},{1.5*cos(40)}) --
	({-1.5*sin(40)+5.7*cos(40)},{1.5*cos(40)+5.7*sin(40)});
	
	\draw[thick,->,domain=54:-54,samples=181,variable=\t]
	plot ({\rturn*cos(110+\t)/cos(9*\t/7)},
	{\rturn*sin(110+\t)/cos(9*\t/7)});
	
	\draw[<->] (-4.1,0) -- (-4.1,1.5)
	node[midway,left] {$b$};
	\node[above] at (-3.8,1.55) {$\omega$};
	\node[right] at (3.05,4.55) {$\xi_\infty(z;\lambda)$};
	
	\draw[densely dashed] (O) -- (C)
	node[pos=0.70,left] {$r_0(b)$};
	\fill (C) circle (1.5pt);
	\draw[gray] ($(C)+(-0.48,-0.175)$) -- ($(C)+(0.48,0.175)$);
	
	\coordinate (D) at ({1.0},{(1.5+sin(40))/cos(40)});
	\draw[densely dotted] (D) -- ++(2.0,0);
	\draw[->] ($(D)+(0.85,0)$)
	arc[start angle=0,end angle=40,radius=0.85];
	\node at ($(D)+(1.77,0.44)$) {$\Theta_\lambda(b)$};
\end{tikzpicture}

	\caption{Radial scattering at the fixed energy $\lambda$. The dashed
	circle $r=r_\lambda$, defined by $v(r_\lambda)=\lambda$, bounds the
	classically forbidden region. For $b>0$, the distance of closest
	approach satisfies $r_0(b)>r_\lambda$. The trajectory is asymptotic
	to the incoming and outgoing free lines as $t\to-\infty$ and
	$t\to+\infty$, respectively.}
\label{fig:radial-scattering}
\end{figure}

\mn 
Since $\rho_\lambda$ is strictly increasing on $[r_\lambda,+\infty)$,
it can be used as a new radial variable. We denote its inverse by
\begin{equation}\label{eq:r-of-rho}
	r=r_\lambda(\rho),
	\qquad
	\rho=\rho_\lambda(r),
\end{equation}
and introduce
\begin{equation}\label{eq:q-rho}
	q_\lambda(\rho)
	:=
	\frac{d}{d\rho}
	\log\frac{r_\lambda(\rho)}{\rho}.
\end{equation}

\mn 
The additional assumption \eqref{eq:radial-convexity}, which is not
needed for the Firsov--Abel inversion itself, provides a convenient
sufficient condition ensuring that the radial scattering geometry has
only one classical branch. More precisely, we have the following result.

\begin{lemma}\label{lem:one-trajectory}
	Assume \eqref{eq:radial-short-range},
	\eqref{eq:radial-keller}, \eqref{eq:radial-energy}, and
	\eqref{eq:radial-convexity}. Then the deflection function
	$b\mapsto\Theta_\lambda(b)$ is strictly decreasing:
	\begin{equation}\label{eq:deflection-monotone}
		\Theta_\lambda'(b)<0,
		\qquad b>0.
	\end{equation}
	Moreover,
	\begin{equation}\label{eq:deflection-limits}
		\lim_{b\to0^+}\Theta_\lambda(b)=\pi,
		\qquad
		\lim_{b\to+\infty}\Theta_\lambda(b)=0.
	\end{equation}
	In particular, for every scattering angle
	$\Theta\in(0,\pi)$, there exists a unique impact parameter
	$b>0$ such that
	\begin{equation}\label{eq:unique-impact-parameter}
		\Theta_\lambda(b)=\Theta.
	\end{equation}
\end{lemma}

\begin{proof}
	Using the change of variables
	$\rho=\rho_\lambda(r)$ in \eqref{eq:deflection}, and recalling that
	$\rho_\lambda(r_0(b))=b$, we have
	\begin{equation}\label{eq:rho-radical}
		\sqrt{
			1-\frac{v(r)}{\lambda}-\frac{b^2}{r^2}
		}
		=
		\frac{\sqrt{\rho^2-b^2}}{r},
	\end{equation}
	and
	\begin{equation}\label{eq:deflection-change-variable}
		\Phi_\lambda(b)
		=
		b\int_b^{+\infty}
		\frac{1}{r_\lambda(\rho)}
		\frac{dr_\lambda}{d\rho}(\rho)
		\frac{d\rho}{\sqrt{\rho^2-b^2}}.
	\end{equation}
	By \eqref{eq:q-rho},
	\begin{equation}\label{eq:q-rho-derivative}
		\frac{1}{r_\lambda(\rho)}
		\frac{dr_\lambda}{d\rho}(\rho)
		=
		\frac{1}{\rho}+q_\lambda(\rho).
	\end{equation}
	Since
	\begin{equation}\label{eq:free-Abel-integral}
		b\int_b^{+\infty}
		\frac{d\rho}{\rho\sqrt{\rho^2-b^2}}
		=
		\frac{\pi}{2},
	\end{equation}
	substitution into
	$\Theta_\lambda(b)=\pi-2\Phi_\lambda(b)$ yields
	\begin{equation}\label{eq:deflection-Abel}
		\Theta_\lambda(b)
		=
		-2b
		\int_b^{+\infty}
		\frac{q_\lambda(\rho)}
		{\sqrt{\rho^2-b^2}}
		\,d\rho.
	\end{equation}
	After the change of variables $\rho=bt$, we obtain
	\begin{equation}\label{eq:deflection-Abel-scaled}
		\Theta_\lambda(b)
		=
		-2b
		\int_1^{+\infty}
		\frac{q_\lambda(bt)}
		{\sqrt{t^2-1}}
		\,dt.
	\end{equation}
	Differentiating with respect to $b$ gives
	\begin{equation}\label{eq:deflection-derivative}
		\Theta_\lambda'(b)
		=
		-2
		\int_1^{+\infty}
		\frac{
			\displaystyle
			\frac{d}{d\rho}
			\bigl(\rho q_\lambda(\rho)\bigr)
			\big|_{\rho=bt}
		}
		{\sqrt{t^2-1}}
		\,dt.
	\end{equation}
	To determine the sign of the integrand, set
	\begin{equation}\label{eq:Q-radial}
		Q_\lambda(r)
		:=
		-\frac{r v'(r)}{\lambda-v(r)},
		\qquad r>r_\lambda.
	\end{equation}
	A direct computation gives
	\begin{equation}\label{eq:Q-derivative}
		Q_\lambda'(r)
		=
		-\frac{
			(\lambda-v(r))(r v'(r))'
			+r(v'(r))^2
		}{
			(\lambda-v(r))^2
		}.
	\end{equation}
By \eqref{eq:radial-convexity}, since $\lambda-v(r)>0$ and $v'(r)<0$, the
numerator in \eqref{eq:Q-derivative} is strictly positive. Hence
\begin{equation}\label{eq:Q-decreasing}
	Q_\lambda'(r)<0,
	\qquad r>r_\lambda.
\end{equation}
	On the other hand, from the definitions of $\rho_\lambda$ and
	$q_\lambda$, one obtains
	\begin{equation}\label{eq:rho-q-Q}
		\rho q_\lambda(\rho)
		=
		\frac{1}{
			1+\frac12 Q_\lambda(r_\lambda(\rho))
		}
		-1.
	\end{equation}
	Since $r_\lambda'(\rho)>0$ and $Q_\lambda'<0$, we conclude that
	\begin{equation}\label{eq:rho-q-positive}
		\frac{d}{d\rho}
		\bigl(\rho q_\lambda(\rho)\bigr)>0.
	\end{equation}
	Inserting \eqref{eq:rho-q-positive} into
	\eqref{eq:deflection-derivative} proves
	\eqref{eq:deflection-monotone}.  Finally, when $b\to0^+$, the turning point $r_0(b)$ tends to
	$r_\lambda$ and the trajectory approaches the radial trajectory.
	Indeed, since $L=\sqrt{2\lambda}\,b$, the limit $b\to0^+$ corresponds to vanishing angular momentum, and
	hence to a radial trajectory which reaches the turning radius
	$r_\lambda$ and is backscattered along the same radial line; hence
	\begin{equation}\label{eq:deflection-small-b}
		\Theta_\lambda(b)\longrightarrow\pi,
		\qquad b\to0^+.
	\end{equation}
	On the other hand, the short-range assumption implies that the
	scattering becomes asymptotically free as $b\to+\infty$, and thus
\begin{equation}\label{eq:deflection-large-b}
	\Theta_\lambda(b)\longrightarrow0,
	\qquad b\to+\infty.
\end{equation}
	This proves \eqref{eq:deflection-limits}. Therefore
	$b\mapsto\Theta_\lambda(b)$ is one-to-one from $(0,+\infty)$ onto
	$(0,\pi)$.
\end{proof}


\subsection{The Robert--Tamura asymptotics}
\label{subsec:Robert-Tamura}

We now recall the semiclassical asymptotics of the scattering
amplitude obtained by Robert and Tamura \cite{RobertTamura}. This
provides the link between the quantum differential cross section and
the classical scattering data introduced in the previous subsection.

\mn 
Fix an incoming direction $\omega\in\mathbb S^{n-1}$ and choose
oriented orthonormal coordinates
$(z_1,\ldots,z_{n-1})$ on the impact plane
\[
\Lambda_\omega=\omega^\perp.
\]
For $z\in\Lambda_\omega$, let
\[
(q_\infty(t;z,\lambda),p_\infty(t;z,\lambda))
\]
be the scattering trajectory defined by the incoming asymptotics
\[
q_\infty(t;z,\lambda)
=
\sqrt{2\lambda}\,\omega t+z+o(1),
\qquad
p_\infty(t;z,\lambda)
=
\sqrt{2\lambda}\,\omega+o(1),
\qquad t\to-\infty.
\]
Assume that $\lambda>0$ is a non-trapping energy. Then there exist
$\xi_\infty(z;\lambda)\in\mathbb S^{n-1}$ and
$r_\infty(z;\lambda)\in\mathbb R^n$ such that
\begin{equation}\label{eq:RT-outgoing-asymptotics}
	\begin{split}
		q_\infty(t;z,\lambda)
		&=
		\sqrt{2\lambda}\,\xi_\infty(z;\lambda)t
		+r_\infty(z;\lambda)+o(1),\\
		p_\infty(t;z,\lambda)
		&=
		\sqrt{2\lambda}\,\xi_\infty(z;\lambda)+o(1),
	\end{split}
	\qquad t\to+\infty.
\end{equation}

\mn 
For fixed $\omega\in\mathbb S^{n-1}$, consider the smooth map
\begin{equation}\label{eq:F-omega}
	F_\omega:\Lambda_\omega\longrightarrow\mathbb S^{n-1},
	\qquad
	F_\omega(z)=\xi_\infty(z;\lambda).
\end{equation}
The angular density associated with the trajectory issued from
$z\in\Lambda_\omega$ is defined by
\begin{equation}\label{eq:angular-density}
	\sigma(z;\lambda)
	=
	\left|
	\det\left(
	\xi_\infty(z;\lambda),
	\partial_{z_1}\xi_\infty(z;\lambda),
	\ldots,
	\partial_{z_{n-1}}\xi_\infty(z;\lambda)
	\right)
	\right|.
\end{equation}
Thus $\sigma(z;\lambda)$ is the absolute value of the Jacobian of
$F_\omega$ between $\Lambda_\omega$ and $\mathbb S^{n-1}$.

\mn
An outgoing direction $\theta\in\mathbb S^{n-1}$,
$\theta\neq\omega$, is said to be regular for the incoming direction
$\omega$ if
\begin{equation}\label{eq:regular-direction}
	\sigma(z;\lambda)\neq0
	\quad\text{for all }z\in\Lambda_\omega
	\text{ such that }F_\omega(z)=\theta.
\end{equation}

\mn 
Assume that $\theta$ is regular. Since $\theta\neq\omega$ and
$\xi_\infty(z;\lambda)\to\omega$ as $|z|\to+\infty$, the set
\[
F_\omega^{-1}(\theta)
=
\{z\in\Lambda_\omega;\ \xi_\infty(z;\lambda)=\theta\}
\]
is compact. By the inverse function theorem, its points are isolated,
and hence this set is finite. We denote its cardinality by
$\ell(\theta,\lambda)$ and its elements by
\begin{equation}\label{eq:RT-branches}
	F_\omega^{-1}(\theta)
	=
	\{w_1(\theta;\lambda),\ldots,
	w_{\ell(\theta,\lambda)}(\theta;\lambda)\}.
\end{equation}
Each $w_j(\theta;\lambda)$ determines a classical scattering
trajectory with incoming direction $\omega$ and outgoing direction
$\theta$. For each such trajectory, Robert and Tamura associate the action
\begin{equation}\label{eq:RT-action}
	\begin{split}
		\mathcal A_j(\omega,\theta;\lambda)
		&=
		\int_{-\infty}^{+\infty}
		\left(
		\frac{|p_\infty(t;w_j,\lambda)|^2}{2}
		-V(q_\infty(t;w_j,\lambda))
		-\lambda
		\right)\diff t
		\\
		&\quad
		-
		\left\langle
		r_\infty(w_j;\lambda),
		\sqrt{2\lambda}\,\theta
		\right\rangle ,
	\end{split}
\end{equation}
where $w_j=w_j(\theta;\lambda)$.
The integral in \eqref{eq:RT-action} represents the difference
between the interacting and free actions along trajectories which
are asymptotic as $t\to-\infty$. Let
$\mu_j\in\mathbb Z$ denote the Keller--Maslov--Morse index of the
corresponding trajectory.

\mn 
With the normalization of the scattering amplitude introduced above,
the Robert--Tamura formula \cite{RobertTamura} reads
\begin{equation}\label{eq:RT-asymptotics-radial}
	f(\omega\to\theta;\lambda,h)
	=
	\sum_{j=1}^{\ell(\theta,\lambda)}
	\sigma(w_j;\lambda)^{-1/2}
	\exp\left(
	\frac{\ii}{h}\mathcal A_j(\omega,\theta;\lambda)
	-\frac{\ii\pi}{2}\mu_j
	\right)
	+O(h),
	\qquad h\to0.
\end{equation}
More precisely, under the same assumptions, the scattering amplitude
admits a complete asymptotic expansion in powers of $h$.

\mn
In the radial setting, the monotonicity of the deflection function
also implies that all non-forward and non-backward scattering
directions are regular.

\begin{lemma}\label{lem:radial-regular-directions}
	Under the assumptions of Lemma~\ref{lem:one-trajectory}, every
	outgoing direction
	\[
	\theta\in\mathbb S^{n-1}\setminus\{\omega,-\omega\}
	\]
	is regular for the incoming direction $\omega$. Moreover, there is
	exactly one classical trajectory connecting $\omega$ to $\theta$,
	that is,
	\[
	\ell(\theta,\lambda)=1.
	\]
\end{lemma}

\begin{proof}
	Fix
	\begin{equation}\label{eq:radial-theta}
		\theta\in\mathbb S^{n-1}\setminus\{\omega,-\omega\},
	\end{equation}
	and set
	\begin{equation}\label{eq:scattering-angle-theta}
		\Theta
		=
		\arccos(\omega\cdot\theta)
		\in(0,\pi).
	\end{equation}
	By Lemma~\ref{lem:one-trajectory}, there exists a unique
	$b>0$ such that
	\begin{equation}\label{eq:deflection-equals-theta}
		\Theta_\lambda(b)=\Theta.
	\end{equation}
	Define
	\begin{equation}\label{eq:radial-eta-from-theta}
		\eta
		=
		\frac{\theta-(\omega\cdot\theta)\omega}{\sin\Theta}
		\in\mathbb S^{n-2}\subset\omega^\perp,
	\end{equation}
	and set
	\begin{equation}\label{eq:impact-polar}
		z=b\eta.
	\end{equation}
	By rotational invariance and the definition of the deflection
	function,
	\begin{equation}\label{eq:radial-outgoing-direction}
		F_\omega(z)
		=
		\cos\Theta_\lambda(b)\,\omega
		+
		\sin\Theta_\lambda(b)\,\eta
		=
		\theta.
	\end{equation}
	Thus
	\begin{equation}\label{eq:radial-preimage-nonempty}
		F_\omega^{-1}(\theta)\neq\varnothing.
	\end{equation}
	
	\mn
	We now prove that $\theta$ is regular. Let
	$z=b\eta\in F_\omega^{-1}(\theta)$ be arbitrary, with
	$b>0$ and $\eta\in\mathbb S^{n-2}\subset\omega^\perp$.
	Let $h\in T_z\Lambda_\omega=\omega^\perp$, and write
	\[
	h=a\eta+h_\perp,
	\qquad
	h_\perp\in T_\eta\mathbb S^{n-2}.
	\]
	Since
	\[
	b(z)=|z|,
	\qquad
	\eta(z)=\frac{z}{|z|},
	\]
	we have
	\[
	Db(z)[h]=a,
	\qquad
	D\eta(z)[h]=\frac1b\,h_\perp.
	\]
	Hence
	\begin{equation}\label{eq:radial-F-differential}
		D_zF_\omega[h]
		=
		a\Theta_\lambda'(b)
		\bigl(
		-\sin\Theta_\lambda(b)\,\omega
		+
		\cos\Theta_\lambda(b)\,\eta
		\bigr)
		+
		\frac{\sin\Theta_\lambda(b)}{b}\,h_\perp.
	\end{equation}
	By Lemma~\ref{lem:one-trajectory},
	\begin{equation}\label{eq:deflection-nondegeneracy}
		0<\Theta_\lambda(b)<\pi,
		\qquad
		\Theta_\lambda'(b)<0.
	\end{equation}
	In particular,
	\begin{equation}\label{eq:radial-nonzero-coefficients}
		\Theta_\lambda'(b)\neq0,
		\qquad
		\frac{\sin\Theta_\lambda(b)}{b}\neq0.
	\end{equation}
	Moreover,
	\begin{equation}\label{eq:radial-orthogonality}
		-\sin\Theta_\lambda(b)\,\omega
		+
		\cos\Theta_\lambda(b)\,\eta
		\perp
		T_\eta\mathbb S^{n-2}.
	\end{equation}
	It follows from \eqref{eq:radial-F-differential} that
	\begin{equation}\label{eq:radial-F-kernel}
		D_zF_\omega[h]=0
		\quad\Longrightarrow\quad
		a=0,
		\qquad
		h_\perp=0.
	\end{equation}
	Hence $D_zF_\omega$ is injective. Since
	\[
	\dim T_z\Lambda_\omega
	=
	\dim T_\theta\mathbb S^{n-1}
	=
	n-1,
	\]
	it follows that $D_zF_\omega$ is invertible and therefore
	\[
	\sigma(z;\lambda)\neq0.
	\]
	Since $z\in F_\omega^{-1}(\theta)$ was arbitrary, $\theta$ is
	regular.
	
	\mn
	It remains to prove that the preimage is unique. Let
	$z=b\eta\in F_\omega^{-1}(\theta)$. Taking the scalar product of
	\eqref{eq:radial-outgoing-direction} with $\omega$, we obtain
	\begin{equation}\label{eq:cos-deflection-equality}
		\cos\Theta_\lambda(b)
		=
		\cos\Theta.
	\end{equation}
	Since both angles belong to $(0,\pi)$,
	\[
	\Theta_\lambda(b)=\Theta.
	\]
	By Lemma~\ref{lem:one-trajectory}, the impact parameter $b$ is
	uniquely determined. Moreover, from
	\eqref{eq:radial-outgoing-direction},
	\[
	\eta
	=
	\frac{\theta-(\omega\cdot\theta)\omega}{\sin\Theta},
	\]
	so $\eta$ is uniquely determined as well. Hence $z=b\eta$ is the
	unique preimage of $\theta$ under $F_\omega$. Therefore
	$
	\ell(\theta,\lambda)=1.
	$
\end{proof}

\mn
For
$
\theta\in\mathbb S^{n-1}\setminus\{\omega,-\omega\},
$
let $w(\theta;\lambda)$ denote the unique element of
$F_\omega^{-1}(\theta)$. By Lemma~\ref{lem:radial-regular-directions}
and the Robert--Tamura formula,
\begin{equation}\label{eq:RT-cross-section-limit}
	|f(\omega\to\theta;\lambda,h)|^2
	=
	\sigma(w(\theta;\lambda);\lambda)^{-1}
	+
	O(h),
	\qquad h\to0.
\end{equation}
Since, in the one-branch case,
$\sigma(w(\theta;\lambda);\lambda)^{-1}$ is precisely the classical
differential cross section, we obtain
\begin{equation}\label{eq:quantum-classical-cross-section}
	|f(\omega\to\theta;\lambda,h)|^2
	\longrightarrow
	\frac{d\sigma_{\rm cl}}{d\Omega}(\theta;\lambda),
	\qquad h\to0.
\end{equation}
This is the semiclassical information that will be used below to
recover the radial potential.


\subsection{Uniqueness from the differential cross section}
\label{subsec:radial-uniqueness}

Let $V_j(x)=v_j(|x|)$, $j=1,2$, be two radial short-range
potentials satisfying the assumptions above, and let
$r_{j,\lambda}$ be the unique radius defined by
\begin{equation}\label{eq:rj-lambda}
	v_j(r_{j,\lambda})=\lambda.
\end{equation}
We denote by $f_{j,h}(\omega,\theta;\lambda)$ the corresponding
semiclassical scattering amplitudes. Assume that
\begin{equation}\label{eq:cross-section-assumption}
	|f_{1,h}(\omega,\theta;\lambda)|^2
	-
	|f_{2,h}(\omega,\theta;\lambda)|^2
	\longrightarrow0,
	\qquad h\longrightarrow0,
\end{equation}
for every regular scattering angle.

\mn
By the Robert--Tamura asymptotics and the one-trajectory property
established above,
\begin{equation}\label{eq:RT-classical-cross-section}
	|f_{j,h}(\omega,\theta;\lambda)|^2
	\longrightarrow
	\frac{d\sigma_{{\rm cl},j}}{d\Omega}(\theta;\lambda),
	\qquad j=1,2.
\end{equation}
Consequently, \eqref{eq:cross-section-assumption} implies
\begin{equation}\label{eq:classical-cross-sections-equal}
	\frac{d\sigma_{{\rm cl},1}}{d\Omega}(\theta;\lambda)
	=
	\frac{d\sigma_{{\rm cl},2}}{d\Omega}(\theta;\lambda)
\end{equation}
for every $\theta\in\mathbb S^{n-1}\setminus\{\omega,-\omega\}$.

\mn
We first recover the deflection function from the classical
differential cross section. By Lemma~\ref{lem:one-trajectory},
the deflection function $b\mapsto\Theta_\lambda(b)$ is strictly
decreasing from $\pi$ to $0$. Its inverse will be denoted by
\begin{equation}\label{eq:b-of-theta}
	b=b(\Theta),
	\qquad 0<\Theta<\pi.
\end{equation}
The classical differential cross section is given by
(see, for instance, Keller--Kay--Shmoys \cite{KellerKayShmoys1956}
for $n=3$),
\begin{equation}\label{eq:radial-classical-cross-section}
	\frac{d\sigma_{\rm cl}}{d\Omega}(\Theta;\lambda)
	=
	\frac{b(\Theta)^{n-2}}
	{\sin^{n-2}\Theta}
	\left|
	\frac{db}{d\Theta}(\Theta)
	\right|.
\end{equation}
Since $b(\Theta)$ is strictly decreasing,
\eqref{eq:radial-classical-cross-section} yields
\begin{equation}\label{eq:b-cross-section-ode}
	\frac{d}{d\Theta}\left( b(\Theta)^{n-1}\right)
	=
	-(n-1)\sin^{n-2}\Theta\,
	\frac{d\sigma_{\rm cl}}{d\Omega}(\Theta;\lambda).
\end{equation}
Moreover, Lemma~\ref{lem:one-trajectory} gives
\begin{equation}\label{eq:b-backscattering-limit}
	b(\Theta)\longrightarrow0
	\qquad\text{as }\Theta\to\pi^-.
\end{equation}
Integrating \eqref{eq:b-cross-section-ode} from $\Theta$ to $\pi$,
we obtain
\begin{equation}\label{eq:b-from-cross-section}
	b(\Theta)^{n-1}
	=
	(n-1)
	\int_\Theta^\pi
	\sin^{n-2}\alpha\,
	\frac{d\sigma_{\rm cl}}{d\Omega}(\alpha;\lambda)
	\,d\alpha.
\end{equation}
Thus the classical differential cross section uniquely determines
$b(\Theta)$, and hence its inverse $\Theta_\lambda(b)$. It follows
from \eqref{eq:classical-cross-sections-equal} that
\begin{equation}\label{eq:deflection-functions-equal}
	\Theta_{1,\lambda}(b)
	=
	\Theta_{2,\lambda}(b),
	\qquad b>0.
\end{equation}

\mn
\mn
It remains to recover the potential from the deflection function.
For this purpose, we use the classical Firsov--Abel inversion procedure
of Firsov~\cite{Firsov1953} and Keller--Kay--Shmoys~\cite{KellerKayShmoys1956};
see also Jollivet~\cite{JollivetRadial} for a related treatment of
short-range radial potentials. We briefly recall the argument in our
notation. By \eqref{eq:deflection-Abel},
\begin{equation}\label{eq:radial-deflection-Abel-recall}
	\Theta_\lambda(b)
	=
	-2b
	\int_b^{+\infty}
	\frac{q_\lambda(\rho)}
	{\sqrt{\rho^2-b^2}}
	\,d\rho,
\end{equation}
where
\begin{equation}\label{eq:radial-q-recall}
	q_\lambda(\rho)
	=
	\frac{d}{d\rho}
	\log\frac{r_\lambda(\rho)}{\rho}.
\end{equation}
Set
\begin{equation}\label{eq:radial-F-definition}
	F_\lambda(\rho)
	=
	\log\frac{r_\lambda(\rho)}{\rho}.
\end{equation}
Then $q_\lambda=F_\lambda'$, and
\eqref{eq:radial-deflection-Abel-recall} becomes
\begin{equation}\label{eq:deflection-Abel-F}
	\frac{\Theta_\lambda(b)}{b}
	=
	-2
	\int_b^{+\infty}
	\frac{F_\lambda'(\rho)}
	{\sqrt{\rho^2-b^2}}
	\,d\rho.
\end{equation}
After the quadratic change of variables $x=b^2$, $y=\rho^2$,
this reduces to the standard Abel transform. Applying the classical
Abel inversion formula and using
\begin{equation}\label{eq:F-lambda-infinity}
	F_\lambda(\rho)
	=
	O(\rho^{-\mu}),
	\qquad \rho\to+\infty,
\end{equation}
we obtain
\begin{equation}\label{eq:Firsov-Abel-inversion}
	F_\lambda(\rho)
	=
	\frac{1}{\pi}
	\int_\rho^{+\infty}
	\frac{\Theta_\lambda(b)}
	{\sqrt{b^2-\rho^2}}
	\,db,
	\qquad \rho>0.
\end{equation}
\mn
Recalling the definition of $F_\lambda$, we obtain
\begin{equation}\label{eq:r-lambda-from-deflection}
	r_\lambda(\rho)
	=
	\rho
	\exp\left(
	\frac{1}{\pi}
	\int_\rho^{+\infty}
	\frac{\Theta_\lambda(b)}
	{\sqrt{b^2-\rho^2}}
	\,db
	\right),
	\qquad \rho>0.
\end{equation}

\mn
Applying \eqref{eq:r-lambda-from-deflection} to $V_1$ and $V_2$ and
using \eqref{eq:deflection-functions-equal}, we obtain
\begin{equation}\label{eq:inverse-optical-radii-equal}
	r_{1,\lambda}(\rho)
	=
	r_{2,\lambda}(\rho),
	\qquad \rho>0.
\end{equation}
Since the functions $r_{1,\lambda}$ and $r_{2,\lambda}$ coincide on
$(0,+\infty)$, their ranges coincide. Since
\begin{equation}\label{eq:inverse-optical-radius-range}
	r_{j,\lambda}\big((0,+\infty)\big)
	=
	(r_{j,\lambda},+\infty),
\end{equation}
their left endpoints coincide, and hence
\begin{equation}\label{eq:turning-radii-equal}
	r_{1,\lambda}
	=
	r_{2,\lambda}
	=:r_\lambda.
\end{equation}
Taking inverse functions in
\eqref{eq:inverse-optical-radii-equal}, we then obtain
\begin{equation}\label{eq:optical-radii-equal}
	\rho_{1,\lambda}(r)
	=
	\rho_{2,\lambda}(r),
	\qquad r>r_\lambda.
\end{equation}
By continuity, the equality also holds at $r=r_\lambda$. Finally, from
\begin{equation}\label{eq:potential-from-optical-radius}
	v_j(r)
	=
	\lambda
	\left(
	1-\frac{\rho_{j,\lambda}(r)^2}{r^2}
	\right),
	\qquad r\geq r_\lambda,
\end{equation}
we conclude that
\begin{equation}\label{eq:radial-potential-uniqueness}
	v_1(r)
	=
	v_2(r),
	\qquad r\geq r_\lambda.
\end{equation}
Equivalently,
\begin{equation}\label{eq:radial-potential-uniqueness-x}
	V_1(x)
	=
	V_2(x),
	\qquad |x|\geq r_\lambda.
\end{equation}
Thus the semiclassical differential cross section at the fixed energy
$\lambda$ determines the radial potential throughout the classically
accessible region.


\section{An extension to  metric perturbations}
\label{sec:metric-perturbations}

\mn
We conclude with an extension of the preceding argument to compactly
supported perturbations of the Euclidean metric. The proof follows the
same general strategy as in the potential case, and we therefore only
give its main steps. 

\mn 
Let $g$ be a smooth
Riemannian metric on $\mathbb R^n$, $n\geq3$, such that
\begin{equation}\label{eq:metric-euclidean-outside}
	g=g_0
	\qquad\text{outside a compact set},
\end{equation}
where $g_0$ denotes the Euclidean metric. We consider the
semiclassical Laplace--Beltrami operator
\begin{equation}\label{eq:metric-operator}
	P_g(h)
	=
	-\frac{h^2}{2}\Delta_g
\end{equation}
and denote by $S_{g,h}(\lambda)$ its scattering matrix at the fixed
energy $\lambda>0$, with the Euclidean Laplacian as the free reference
operator.

\mn
The scattering matrix is invariant under compactly supported
changes of coordinates. More precisely, let
$T:\mathbb R^n\to\mathbb R^n$ be a smooth diffeomorphism which
equals the identity outside a compact set, and let
$\widetilde g=T^*g$. Then
\begin{equation}\label{eq:metric-scattering-coordinate-invariance}
	S_{\widetilde g,h}(\lambda)=S_{g,h}(\lambda).
\end{equation}
This coordinate invariance is established
in \cite[Proposition 2.4]{CottaRamusinoKrugerSchrader}.\footnote{In fact, the scattering matrix is invariant under diffeomorphisms $T : \R^n \to \R^n$ such that $|T(x)-x|=o(1)$ and $|\partial_x T(x)| = 1 + o(1)$ as $|x| \to \infty$.}
Hence any uniqueness result based on the scattering matrix can hold at most modulo such changes of coordinates.

\mn
Assume that the geodesic flow of $g$ is non-trapping. Since $g$ is
Euclidean outside a compact set, every geodesic has incoming and
outgoing Euclidean asymptotic data. As in the potential case, these
data define the classical scattering map
\begin{equation}\label{eq:metric-classical-scattering-map}
	\kappa_{g,\lambda}
	:
	T^*\mathbb S^{n-1}
	\longrightarrow
	T^*\mathbb S^{n-1}.
\end{equation}

\mn
Although Ingremeau formulates his Gaussian-state propagation result for potential scattering, he points out that the same result holds for compactly supported metric perturbations of the Laplacian, with
a similar proof \cite{Ingremeau}. Consequently, the proof of
Proposition~\ref{lem:recovery-kappa} applies in the metric setting as well. Note also that the relation between the semiclassical scattering matrix and the classical scattering relation has been studied by
Hassell and Wunsch \cite{HassellWunsch2008} for non-trapping
scattering metrics.

\begin{lemma}\label{lem:metric-recovery-kappa}
	Let $g_1$ and $g_2$ be smooth Riemannian metrics on
	$\mathbb R^n$, $ n \geq 2$, which coincide with the Euclidean metric outside
	a compact set, and assume that their geodesic flows are
	non-trapping. If, for some fixed $\lambda>0$,
	\begin{equation}\label{eq:metric-scattering-close}
		\liminf_{h\to0}
		\|S_{g_1,h}(\lambda)-S_{g_2,h}(\lambda)\|
		_{\mathcal B(L^2(\mathbb S^{n-1}))}
		<
		2,
	\end{equation}
	then
	\begin{equation}\label{eq:metric-scattering-maps-equal}
		\kappa_{g_1,\lambda}
		=
		\kappa_{g_2,\lambda}.
	\end{equation}
\end{lemma}

\mn 
Choose $R>0$ so that both metrics are Euclidean near and
outside $\partial B(0,R)$. Since geodesics are straight lines
outside this ball, their asymptotic data determine their
entry and exit points and directions at the boundary.
Thus, equality of the scattering maps in Lemma~\ref{lem:metric-recovery-kappa}
implies equality of the boundary scattering relations defined in \eqref{eq:SUV-scattering-relation}.
To apply the lens rigidity theorem of Stefanov--Uhlmann--Vasy,
we also need equality of the lengths of the corresponding
geodesics inside the ball. This follows from the next lemma. The same conclusion was also recently obtained by Cekić, Krupchyk,
Sahoo, and Uhlmann \cite[Proposition 2.6]{CekicKrupchykSahooUhlmann2026}
under the more general assumptions that the metrics are
nontrapping, have strictly convex boundary, and agree on
$TM|_{\partial M}$.

\begin{lemma}\label{lem:scattering-lengths}
	Let $n\geq 2$, and let $g_1$ and $g_2$ be two smooth
	nontrapping Riemannian metrics on $M=\overline{B(0,R)}$,
	equal to the Euclidean metric near $\partial M$.
	If their scattering relations coincide, then their
	lens data coincide.
\end{lemma}

	\begin{proof}
		The proof follows the first-variation argument used by Wen
		to show that, for simple Riemannian surfaces, equality of
		scattering data implies equality of lens data
		\cite[Section 1.3]{Wen}. 
		Let \(\partial_- SM\) be the set of all unit vectors at \(\partial M\) that point strictly inward.
		For a point \((x, v) \in \partial_- SM\), let \(\ell_j(x, v)\) be the length of the \(g_j\)-geodesic starting at \(x\) with initial velocity \(v\), measured until it first exits \(M\). The functions \(\ell_j\) are smooth because both metrics are non-trapping, the boundary is strictly convex, (the metrics match the Euclidean metric near the boundary).
		
		\mn
	Let $s\mapsto(x(s),v(s))$ be any smooth curve in
	$\partial_-SM$. Here $s$ parametrizes a family of
	initial data.
	Since the scattering relations are equal, the two
	geodesics corresponding to each $s$ have the same
	exit point $y(s)$ and exit direction $w(s)$.
	The first variation of length gives
	\begin{equation}
		\frac{d}{ds}\ell_j(x(s),v(s))
		=
		\langle w(s),y'(s)\rangle
		-
		\langle v(s),x'(s)\rangle,
		\qquad j=1,2.
	\end{equation}
	Only the endpoints contribute to this formula,
	since the curves are geodesics. Both endpoints lie
	on $\partial M$, where the two metrics are Euclidean. The right-hand side is the same for $j=1,2$.
	Since the curve of initial data is arbitrary, we obtain
	$d(\ell_1-\ell_2)=0$.
	Moreover, $\partial_-SM$ is connected for $n\geq2$.
	Thus $\ell_1-\ell_2$ is constant.
	
	\mn
	To determine this constant, choose an inward direction
	sufficiently close to a tangent direction at the boundary.
	The corresponding geodesics are then the same short
	Euclidean chord, contained in the neighborhood of
	$\partial M$ where both metrics are Euclidean.
	Their lengths are equal, so the constant is zero.
	Therefore $\ell_1=\ell_2$, which proves the lemma.
\end{proof}

\mn
We now apply the lens rigidity theorem of
Stefanov--Uhlmann--Vasy \cite[Theorem 8.1]{SUV2021}.
We assume that one of the metrics admits a smooth strictly
convex function. Here, strict convexity means that the
Riemannian Hessian of this function is positive definite.
In the global Cartesian coordinates on $\mathbb R^n$,
this Hessian is given by
\begin{equation}\label{eq:metric-hessian}
	(\operatorname{Hess}_g f)_{ij}
	=
	\partial_i\partial_j f
	-
	\sum_{k=1}^n \Gamma^k_{ij}(g)\,\partial_k f,
	\qquad 1\leq i,j\leq n,
\end{equation}
where
\begin{equation}
\Gamma^k_{ij}(g)
=
\frac12\sum_{\ell=1}^n g^{k\ell}
\bigl(
\partial_i g_{j\ell}
+\partial_j g_{i\ell}
-\partial_\ell g_{ij}
\bigr)
\end{equation}
are the Christoffel symbols of $g$, and $(g^{k\ell})$
is the inverse matrix of $(g_{k\ell})$.
For the Euclidean metric, the Christoffel symbols vanish,
and \eqref{eq:metric-hessian} reduces to the usual Hessian.

\begin{theorem}\label{thm:metric-simple-uniqueness}
	Let $g_1$ and $g_2$ be smooth Riemannian metrics on
	$\mathbb R^n$, $n\geq3$, with nontrapping geodesic flows.
	Assume that there exists $R>0$ such that
	\begin{equation}\label{eq:metric-support-ball}
		\operatorname{supp}(g_j-g_0)\Subset B(0,R),
		\qquad j=1,2.
	\end{equation}
	Set $M=\overline{B(0,R)}$.
	Assume also that there exists $f\in C^\infty(M)$ such that
	\begin{equation}\label{eq:metric-convex-function}
		\operatorname{Hess}_{g_1}f>0
		\quad\text{on }M,
		\qquad
		f^{-1}(0)=\partial M.
	\end{equation}
	If, for some fixed $\lambda>0$,
	\begin{equation}\label{eq:metric-main-assumption}
		\liminf_{h\to0}
		\|S_{g_1,h}(\lambda)-S_{g_2,h}(\lambda)\|
		_{\mathcal B(L^2(\mathbb S^{n-1}))}
		<
		2,
	\end{equation}
then there exists a smooth diffeomorphism
\[
\Psi:\mathbb R^n\longrightarrow\mathbb R^n,
\qquad
\Psi=\operatorname{Id}
\quad\text{outside }M,
\]
such that
\begin{equation}\label{eq:metric-isometry}
g_1=\Psi^*g_2
\quad\text{on }\mathbb R^n.
\end{equation}
\end{theorem}

\mn
The condition on $f$ is imposed only on $g_1$.
The function $f$ has a unique critical point in the interior
of $M$, where it attains its minimum.
This is allowed in \cite[Theorem 8.1]{SUV2021}.
No assumption on conjugate points is needed.
The relation between strictly convex functions and convex
foliations is explained in \cite[Section 8]{SUV2021}.

\mn
We give some examples before proving the theorem.
The Hessian computations for Examples~2 and~3 are given in
Appendix~\ref{app:radial-hessians}.

\begin{example}[Euclidean metric and small perturbations]
	For $g_1=g_0$, take
	\[
	f(x)=|x|^2-R^2.
	\]
	Then
	\[
	\operatorname{Hess}_{g_0}f=2g_0>0,
	\qquad
	f^{-1}(0)=\partial M.
	\]
	The same function works for any metric $g_1$ sufficiently
	close to $g_0$ in $C^1(M)$.
	Indeed, in Euclidean coordinates,
	\[
	(\operatorname{Hess}_{g_1}f)_{ij}
	=
	2\delta_{ij}
	-
	2\sum_{k=1}^n\Gamma^k_{ij}(g_1)x_k,
	\]
	which remains positive definite when $g_1$ is sufficiently
	close to $g_0$ in $C^1(M)$.
	This example includes perturbations which are not radial.
\end{example}

\begin{example}[Radial conformal metrics]
	Let
	\[
	g_1=c(r)^{-2}g_0,
	\qquad r=|x|,
	\]
	where $c$ is smooth and positive, constant near $r=0$,
	and equal to $1$ near and outside $r=R$.
	Assume the Herglotz condition
	\begin{equation}\label{eq:metric-herglotz}
		\frac{d}{dr}\left(\frac{r}{c(r)}\right)>0,
		\qquad 0<r\leq R.
	\end{equation}
	This condition expresses strict convexity of the Euclidean
	spheres for the metric $g_1$; see
	\cite[Proposition 6.1]{SUV2016} and
	\cite[Section 8]{SUV2021}. In this radial case, we can give the function explicitly:
	\[
	f(x)=-\int_{|x|}^R\frac{s}{c(s)^2}\,ds.
	\]
	It is smooth at the origin, since $c$ is constant there,
	and $f^{-1}(0)=\partial M$.
	A direct computation gives
	\[
	\operatorname{Hess}_{g_1}f
	=
	\left(1-\frac{rc'(r)}{c(r)}\right)g_1.
	\]
	Thus \eqref{eq:metric-herglotz} implies
	\eqref{eq:metric-convex-function}.
	In particular, this example does not require the metric
	to be close to the Euclidean one.
\end{example}

\begin{example}[Radial warped metrics]
	In polar coordinates, consider
	\[
	g_1=dr^2+a(r)^2g_{\mathbb S^{n-1}},
	\]
	where $g_{\mathbb S^{n-1}}$ is the standard sphere metric.
	Assume that $a$ is smooth, that $a(r)=r$ near $r=0$
	and near and outside $r=R$, and that
	\[
	a(r)>0,\qquad a'(r)>0,\qquad r>0.
	\]
	The metric is then smooth at the origin and Euclidean
	outside a compact subset of $B(0,R)$. For $f(x)=|x|^2-R^2$, direct computation gives, for $r>0$,
	\[
	\operatorname{Hess}_{g_1}f
	=
	2\,dr^2+2r\,a(r)a'(r)g_{\mathbb S^{n-1}}>0.
	\]
	Near the origin, this is simply $2g_0$.
	For instance, one may take
	\[
	a(r)=r+\varepsilon\chi(r),
	\qquad
	\chi\in C_c^\infty((0,R)),
	\]
	provided
	\[
	|\varepsilon|\,\|\chi'\|_\infty<1.
	\]
\end{example}

\mn
There are also geometric conditions which ensure the
existence of a suitable function.
By \cite[Corollary 1.1 and Section 8]{SUV2021}, this is
the case for a compact manifold with strictly convex
boundary if it is simply connected and has non-positive
sectional curvature, if it is simply connected and has
no focal points, or if it has non-negative sectional
curvature.
Thus the geometric conditions used in the simple case
are covered by the same lens rigidity result.

\mn
For completeness, the convexity assumption already implies
that $g_1$ is nontrapping in $M$.
Indeed, compactness gives a constant $c_0>0$ such that,
along any unit-speed $g_1$-geodesic contained in $M$,
\[
\frac{d^2}{dt^2}f(\gamma(t))
=
\operatorname{Hess}_{g_1}f(\dot\gamma(t),\dot\gamma(t))
\geq c_0.
\]
Such a geodesic cannot stay in $M$ for all positive time,
since $f$ is bounded on $M$.
Outside the ball, geodesics are straight lines.
We have kept the nontrapping assumption for both metrics
in the statement to make the use of
Lemma~\ref{lem:metric-recovery-kappa} explicit.

\begin{proof}[Proof of Theorem~\ref{thm:metric-simple-uniqueness}]
By Lemma~\ref{lem:metric-recovery-kappa}, equation~\eqref{eq:metric-main-assumption} gives
$
\kappa_{g_1,\lambda}=\kappa_{g_2,\lambda}.
$
Since both metrics are Euclidean outside \(M\), the explanation following Lemma~\ref{lem:metric-recovery-kappa} implies that the scattering relations on \(M\) are equal. Lemma~\ref{lem:scattering-lengths} then gives
$
\ell_{g_1}=\ell_{g_2}.
$
	Hence the two metrics have the same lens data. The boundary of $M$ is strictly convex for both metrics,
	since they are Euclidean near $\partial M$.
	Together with \eqref{eq:metric-convex-function}, these are
	the assumptions of \cite[Theorem 8.1]{SUV2021}.
	That theorem gives a diffeomorphism $\psi:M\to M$,
	equal to the identity on $\partial M$, such that
	$g_1=\psi^*g_2$. It remains to note that $\psi$ is the identity near $\partial M$.
	Since both metrics equal $g_0$ in a neighborhood of $\partial M$
	and $\psi|_{\partial M}=\operatorname{Id}$, after shrinking that
	neighborhood we have
	$
	g_0=\psi^*g_0.
	$
	Thus $\psi$ is a local Euclidean isometry there. Since it fixes
	$\partial M$ pointwise, it is the identity in a one-sided neighborhood of
	$\partial M$. We may therefore extend $\psi$ by the identity to a
	smooth diffeomorphism $\Psi$ of $\mathbb R^n$. Since
	$g_1=g_2=g_0$ outside $M$, the identity
	$g_1=\Psi^*g_2$ holds on all of $\mathbb R^n$.
\end{proof}

\mn
Taking the Euclidean metric as the reference metric gives
the following consequence.

\begin{corollary}\label{cor:metric-euclidean-rigidity}
	Let $g$ be a smooth Riemannian metric on $\mathbb R^n$,
	$n\geq3$, with nontrapping geodesic flow.
	Assume that, for some $R>0$,
	\begin{equation}\label{eq:metric-euclidean-support}
		\operatorname{supp}(g-g_0)\Subset B(0,R).
	\end{equation}
If, for some fixed $\lambda>0$,
\begin{equation}\label{eq:metric-euclidean-scattering}
	\liminf_{h\to0}
	\|S_{g,h}(\lambda)-S_{g_0,h}(\lambda)\|
	_{\mathcal B(L^2(\mathbb S^{n-1}))}
	<2,
\end{equation}
then there exists a smooth diffeomorphism
\[
\Psi:\mathbb R^n\longrightarrow\mathbb R^n,
\qquad
\Psi=\operatorname{Id}
\quad\text{outside }B(0,R),
\]
such that
\begin{equation}\label{eq:metric-euclidean-isometry}
g=\Psi^*g_0
\quad\text{on }\mathbb R^n.
\end{equation}
\end{corollary}

\begin{proof}
	Applying Theorem~\ref{thm:metric-simple-uniqueness}
	with $g_1=g_0$, $g_2=g$, and
	$f(x)=|x|^2-R^2$, we obtain a diffeomorphism
	$\Phi:\mathbb R^n\to\mathbb R^n$, equal to the identity
	outside $M$, such that $g_0=\Phi^*g$.
	Setting $\Psi=\Phi^{-1}$, we have $\Psi=\operatorname{Id}$
	outside $M$ and
	$
	g=\Psi^*g_0.
	$
\end{proof}

\begin{remark}\label{rem:metric-nonsimple}
	The argument uses the full lens data, rather than the
	boundary distance function.
	The common Euclidean neighborhood of the boundary allows
	us to recover the geodesic lengths from the scattering
	relation by Lemma~\ref{lem:scattering-lengths}.
	Thus simplicity is not needed in the theorem or in the
	Euclidean corollary.
\end{remark}

\begin{remark}
	Mu\~noz-Thon \cite{MunozThon2024} studies scattering rigidity
	in the more general setting of simple magnetic-potential systems.
	In the purely Riemannian case, his results include metrics
	in a fixed conformal class and real-analytic metrics.
	Here we consider only Riemannian metrics, but we do not
	assume simplicity. We use instead a strictly convex function
	and the lens rigidity theorem of Stefanov--Uhlmann--Vasy.
	Moreover, our assumption on the semiclassical scattering matrices
	is quantitative:  we only require
	$
	\liminf_{h\to0}
	\|S_{g_1,h}(\lambda)-S_{g_2,h}(\lambda)\|_
	{\mathcal B(L^2(\mathbb S^{n-1}))}
	<2.
	$
\end{remark}

\appendix
\section{Hessian computations for the radial examples}
\label{app:radial-hessians}

We use the coordinate formula
\[
(\operatorname{Hess}_g f)_{ij}
=
\partial_i\partial_j f-\sum_k\Gamma^k_{ij}\partial_k f.
\]
In Example~2, $g_1=c(r)^{-2}g_0$ and $f'(r)=r/c(r)^2$.
For $r>0$, the Christoffel symbols in Cartesian coordinates are
\[
\Gamma^k_{ij}
=
-\frac{c'(r)}{rc(r)}
\bigl(x_i\delta_{jk}+x_j\delta_{ik}-x_k\delta_{ij}\bigr).
\]
Since
\[
\partial_i f=\frac{x_i}{c(r)^2},
\qquad
\partial_i\partial_j f
=
\frac{\delta_{ij}}{c(r)^2}
-\frac{2c'(r)}{rc(r)^3}x_ix_j,
\]
substitution gives
\[
\operatorname{Hess}_{g_1}f
=
\left(1-\frac{rc'(r)}{c(r)}\right)g_1.
\]
The coefficient is positive by the Herglotz condition, since
$1-rc'/c=c(r)(r/c(r))'$.

\mn
In Example~3, write $g_1=dr^2+a(r)^2\gamma$, where
$\gamma=g_{\mathbb S^{n-1}}$, and let $f=r^2-R^2$.
With angular indices $\alpha,\beta$, the relevant Christoffel symbols are
\[
\Gamma^r_{rr}=\Gamma^r_{r\alpha}=0,
\qquad
\Gamma^r_{\alpha\beta}=-a(r)a'(r)\gamma_{\alpha\beta}.
\]
Since $\partial_r f=2r$ and the angular derivatives of $f$ vanish,
\[
\operatorname{Hess}_{g_1}f
=
2\,dr^2+2r\,a(r)a'(r)\gamma,
\]
which is positive definite for $r>0$, since $a(r)>0$ and $a'(r)>0$.

\mn
These computations extend smoothly to the origin:
since $c$ is constant there in Example~2 and $a(r)=r$
there in Example~3, the Hessians near the origin are
respectively $g_1$ and $2g_0$.

\section*{Acknowledgements}

The author warmly thanks Niky Kamran for useful discussions
and for pointing out the reference to Bredon's book.

\section*{AI usage disclosure}

The author used AI tools for assistance with language editing,
graphical preparation, and checking the computations in Section 3 and Appendix~A.
The author takes full responsibility for the mathematical content
of the paper.

\end{document}